\documentclass{article}

\usepackage[round]{natbib}
\usepackage{xcolor}
\usepackage{hyperref}
\hypersetup{
    colorlinks=true,
    citecolor=blue
}
\usepackage{amsfonts}
\usepackage{amssymb}
\usepackage{amsmath}
\usepackage{amsthm}
\newtheorem{theorem}{Theorem}
\newtheorem{definition}{Definition}

\newtheorem{corollary}{Corollary}
\newtheorem{lemma}{Lemma}
\newtheorem{remark}{Remark}

\usepackage{bbm}
\usepackage{graphicx}
\usepackage{geometry}

\renewcommand{\le}{\leqslant}
\renewcommand{\ge}{\geqslant}
\renewcommand{\epsilon}{\varepsilon}

\newcommand {\matl}{\left[ \begin{matrix}}
\newcommand {\matr}{\end{matrix}\right]}
\newcommand {\Exp}{ \mathbb E }
\newcommand {\sg}{\operatorname{sgn}}
\renewcommand {\Pr}{ \mathbb P }

\newcommand{\e}{\mathrm e}

\renewcommand{\d}{\mathrm{d}}

\newcommand{\CI}{\operatorname{CI}}

\newcommand{\diam}{\operatorname{diam}}
\newcommand{\up}{\max}
\newcommand{\lw}{\min}

\newcommand{\iid}{\overset{\mathrm{iid}}{\sim}}

\newcommand{\tv}{D_{\operatorname{TV}}}
\newcommand{\tvball}{\mathbb{B}_{\operatorname{TV}}}
\newcommand{\contam}{\mathcal{C}}
\newcommand{\Mrc}{M^{\mathsf{RC}}}
\newcommand{\Nrc}{N^{\mathsf{RC}}}
\newcommand{\CIRC}{\operatorname{CI}^{\mathsf{R}}}
\DeclareMathOperator*{\Expw}{\Exp}
\DeclareMathOperator*{\Prw}{\Pr}

\DeclareMathOperator*{\polylog}{polylog}

\newcommand{\rhatmu}{\widehat{\mu}^{\mathsf{R}}}

\newcommand{\proofs}{\cref{sec:pf}}

\newcommand{\sectionallcap}[1]{\section{#1}}
\usepackage{wrapfig}
\usepackage[capitalize,noabbrev]{cleveref}
\usepackage{authblk}
\begin{document}



\title{Huber-Robust Confidence Sequences\thanks{Accepted at the 26th International Conference on Artificial Intelligence and Statistics (AISTATS 2023)}}

\author[1]{Hongjian Wang}
\author[2]{Aaditya Ramdas}
\affil[1, 2]{Machine Learning Department, Carnegie Mellon University} 
\affil[2]{Department of Statistics and Data Science, Carnegie Mellon University}
\affil[ ]{\texttt{ \{hjnwang,aramdas\}@cmu.edu  }}

\date{\today} 

\maketitle

\begin{abstract}
  Confidence sequences are confidence intervals that can be sequentially tracked, and are valid at arbitrary data-dependent stopping times. This paper presents confidence sequences for a univariate mean of an unknown distribution with
  %variance (at most) $\sigma^2$, but
  a known upper bound on the $p$-th central moment ($p>1$), but
%an $\varepsilon$ fraction of points is adversarially corrupted,
allowing for (at most) $\varepsilon$ fraction of arbitrary distribution corruption,
as in Huber's contamination model. We do this by designing new robust exponential supermartingales, and show that the resulting confidence sequences attain the optimal width achieved in the nonsequential setting. Perhaps surprisingly, the constant margin between our sequential result and the lower bound is smaller than even fixed-time robust confidence intervals based on the trimmed mean, for example.
%Beyond the rate, our constants are quite small, meaning that our confidence sequences are empirically much tighter than even fixed-time confidence intervals based on median-of-means, for example. 
Since confidence sequences are a common tool used within A/B/n testing and bandits, these results open the door to sequential experimentation that is robust to outliers and adversarial corruptions. 
\end{abstract}

\sectionallcap{Introduction}
In this paper, we study the problem of robust, sequential mean estimation; we are not just interested in producing a point estimator of the mean (of which there are many robust ones), but in quantifying uncertainty in a manner that is both theoretically optimal and practically tight. Let $P$ be an unknown distribution over $\mathbb R$ with a known upper bound $\sigma^2$ on the variance, and unknown mean $\mu$ that we want to estimate. Let $Q$ be another distribution, seen as a ``corruption'' of $P$, such that its total variation (TV) distance from $P$ is at most $\varepsilon$. We assume that an infinite stream of data $X_1, X_2, \dots$ is generated i.i.d.\ according to the corrupted distribution $Q$. The task that this paper shall focus on, is the derivation of a robust \emph{confidence sequence} (CS) \citep{darling1967confidence} for $\mu$, which is a sequence of confidence intervals $\{ \CI_t \}_{t \in \mathbb N^+}$ that guarantees that
\begin{equation}\label{eqn:cs-def}
\forall \, \text{stopping times }\tau > 0, \      \Pr[ \mu \in  \CI_\tau   ] \ge 1-\alpha,
\end{equation}
where $\alpha$ is a predefined miscoverage tolerance.  \cite{howard2021time} showed that condition~\eqref{eqn:cs-def} is equivalent to the time-uniform coverage condition
\begin{equation}\label{eqn:cs-def-unif}
     \Pr[ \forall t \in \mathbb N^+, \  \mu \in  \CI_t   ] \ge 1-\alpha.
\end{equation}
Both (equivalent) guarantees are of course much stronger than fixed-time confidence intervals for $\mu$, for which the sample size must be fixed in advance. Instead a CS allows for sequentially tracking the mean and stopping at any data-dependent time, while still having correct inference. Note that the covered parameter $\mu$ is the mean of the uncorrupted $P$, while the mean of the corrupted, data-generating $Q$ may be arbitrarily away from $\mu$, or even undefined. Also, the variance of $Q$ need not exist.

Confidence sequences are increasingly common in sequential experimentation in the IT industry. Even though they are cast in terms of estimation, they can be used to define anytime-valid p-values~\citep{johari2015always,howard2021time} for composite hypotheses like testing if $\mu \le 0$. CSs are useful for A/B testing or multi-armed bandit testing~\citep{yang2017framework,howard2022sequential}. In fact, they are explicitly used in internal tools or external services of Adobe~\citep{AdobeExperimentation}, Amazon~\citep{EvidentlyAmazon} and Netflix~\citep{lindon2022rapid}, for example. CSs are an integral part of multi-armed bandit algorithms, right from the original papers on regret minimization~\citep[Section 4]{lai1985asymptotically} and best-arm identification~\citep{jamieson2014lil}. Thus robust CSs, that are theoretically tight and not practically loose, could have immediate downstream applications.

\subsection{Our Contributions}

The CS that we shall present in this paper is partly inspired by the non-sequential + non-robust study by \cite{catoni2012challenging}, and the recent sequential + non-robust extension due to \cite{wang2022catoni}. Like
%those
the latter paper, both the upper and lower endpoints of the CS can be seen as M-estimators. In fact, setting $\epsilon=0$ in this work recovers the latter paper's results, which in turn recovers the former paper when interested only in fixed-time CIs under heavy-tailed settings. In this sense, the current paper strengthens the ties between heavy-tailed mean estimation and robust mean estimation, two goals that are often separately pursued but closely related~\citep{prasad2019unified}. 

Our CS not only works nonparametrically for any distribution $P$ and its corruption $Q$ that satisfy the succinct assumptions we stated in the beginning of this section, but also enjoys near-optimal tightness. To elaborate, the width of our robust CS provably matches the lower bound $\sigma \sqrt{\varepsilon}$ for \emph{any} robust estimation methods, sequential or not. Mathematically, we prove that its width is optimal up to a constant factor (currently $28$) with high probability. Perhaps both lower and upper bound can be improved by constant factors in future work, but the important point is that one does not appear to pay even this constant factor price in practice. This is to be contrasted with the non-sequential recent work by \cite{lugosi2021robust}, whose analysis only yields robust confidence intervals with widths at least $96\sqrt{2}$ times the lower bound. The reason behind it is simple: most works, including the above, are not aiming at tight or practical uncertainty quantification --- they typically design point estimators and prove mathematically that these are ``close'' to the true mean achieving the optimal rate in terms of $\sigma,\epsilon$ and sample size; these rate-optimal theorems can be translated to confidence intervals, but these are loose. It is true that these authors often do not to optimize their constants, but the reality is that if we desire a $(1-\alpha)$-CI using their techniques, we must pay a practical price for their loose analysis. This is because their constants appear in the actual construction of the CI. In contrast, our CS does not have large constants in the actual construction: the factor of 28 only arises in the theoretical analysis. This means that our CS is much tighter in practice (than their method, and than our bounds); see Figure~\ref{fig:vs-lm} for details.

\subsection{Related Work}
\textbf{Robust Statistics.} The coinage of the term ``robust'' in the statistical literature was first due to \cite{box1953non}, albeit more in today's sense of nonparametric statistics. Early pioneers in the development of the concept include \cite{tukey1960survey, huber1964robust, huber1965robust, huber1968robust, huber1973robust,bickel1965some}, and \cite{hampel1968contributions, hampel1971general}. While historically ``loaded with many, sometimes inconsistent, connotations'' \citep{ronchetti2009robust}, the concept of statistical robustness in recent years refers almost exclusively to the one pioneered by \cite{huber1964robust}, i.e.\ the setting where data are subject to a certain level of contamination --- which may either be on the underlying distribution \citep{huber1964robust,maronna2019robust,diakonikolas2019robust}, or directly on the data \citep{lecue2020robust,lugosi2021robust,minsker2021robust}. In either case, valid inference on functionals of the original, uncontaminated distribution is of interest. The definition of robustness (Definition~\ref{def:rci-rcs}) that we shall use in this paper also follows this recent convention.  Also, we discuss the relation to another notion of robustness in \cref{sec:rel}, as the premise for a comparative study.

Prominent ideas in robust estimation include, among others, M-estimators and trimming. The use of M-estimators in robust statistical procedures dates back to \cite{huber1964robust}, 
which achieve robustness by curbing and bounding the influence that individual data points can make on the statistics. A recent work by \cite{bhatt2022minimax}, concurrent with ours, shows that the fixed sample size M-estimator due to \cite{catoni2012challenging} is robustly minimax. On the other hand, trimming refers to the practice of directly discarding outliers \citep{anscombe1960rejection}, and trimmed means have long been known to be robust \citep{bickel1965some}. Recently, a sample splitting variant of the trimmed mean due to \cite{lugosi2021robust} was shown to be robust nonparametrically over all finite variance distributions, which we shall discuss in \cref{sec:trim}. Other recent techniques in obtaining robust estimators include median of means \citep{lerasle2011robust, depersin2022robust}, self-normalization \citep{minsker2021robust}, and filtering~\citep{diakonikolas2019robust}. 

\textbf{Sequential Statistics.} \cite{wald1945sequential} first formulated the problem of sequential statistical testing, as deciding to reject the null or not \emph{every time} a new data point is seen.  The sequential type I error is the probability of \emph{ever} rejecting the null when it is true. Apart from sequential tests, the issue of sequential validity leads to a tapestry of concepts including the anytime valid p-value \citep{johari2015always,howard2021time}, and the e-process \citep{howard2020time, grunwald2020safe, ramdas2022testing}. The confidence sequence~\eqref{eqn:cs-def} \citep{darling1967confidence} is the only concept that focuses on sequential estimation rather than testing, and from which the other tools can be derived. Since~\cite{howard2020time}, the use of nonnegative (super)martingales in conjunction with Ville's inequality (see \cref{sec:ville}) has been increasingly common practice in sequential inference, generalizing \citeauthor{wald1945sequential}'s \citeyearpar{wald1945sequential} likelihood ratios favored by the earlier works, and yielding new tools in nonparametric settings~\citep{howard2021time}. Sequentially valid methods provide a theoretical safeguard against the potential peril of p-hacking. We refer the reader to a recent survey \citep{ramdas2022game} for more details.

There have been a few studies in the literature that address the feasibility of robustifying sequential tests, mostly notably the works by \cite{huber1965robust} and by \cite{quang1985robust}, both of them robustifying the likelihood ratio by censoring. Hence they only apply to relatively simple parametric settings, and it is hard to generalize them into our nonparametric setting --- indeed our class of distributions consists of both continuous and discrete distributions with any support, and so there is no single common reference measure with respect to which one can define likelihood ratios between pairs of distributions in our class, and so we will need to entirely avoid likelihood ratio style methods.

In summary, we believe this is the first work to design a robust confidence sequence in any setting.

\sectionallcap{Preliminaries}

\subsection{Notation and Problem Setup}
Throughout the paper, $\mathbb R$ denotes the set of real numbers and $\mathcal{B}(\mathbb R)$ its Borel $\sigma$-algebra. The diameter of any $I \in \mathcal{B}(\mathbb R)$ is denoted by $\diam( I )$. $\mathbb N$ and $\mathbb N^+$ denotes respectively the set of nonnegative and positive integers. We fix a universal probability space $(\Omega, \mathcal{A}, \Pr)$ and a filtration $\{\mathcal{F}_t\}_{t \in \mathbb N}$ when discussing randomness.

Denote by $\mathcal{M}$ the set of all probability measures over $(\mathbb{R}, \mathcal{B}(\mathbb R))$. Its elements, understood as the distributions of real-valued data, are denoted by upper-case italic letters $P$, $Q$ and so on. The following important nonparametric subsets of $\mathcal{M}$ are of interest. We denote by $\mathcal{M}^1$ the set of all distributions with finite mean on $\mathbb{R}$, where $\mu: \mathcal{M}^1 \to \mathbb R$ denotes the mean functional
\begin{equation}\label{eq:defn-M-p}
    \mu(P) = \int x \, \d P.
\end{equation}
For $p > 1$, denote by $\mathcal{M}^p$ the subset of $\mathcal{M}^1$ of those distributions with that have both means and finite $p$-th moments on $\mathbb{R}$, while $v_p: \mathcal{M}^p \to \mathbb R$ denotes the $p$-th absolute central moment functional
\begin{equation}\label{eq:defn-v-p}
    v_p(P) = \int |x - \mu(P)|^p \, \d P. 
\end{equation}
Finally, we use $
    \mathcal{M}^p_{\kappa} = \{ P \in \mathcal{M}^p: v_p(P) \le \kappa \}
$ to denote distributions with  $p$-th absolute central moment bounded by $\kappa$. The familiar class of distributions with variance bounded by $\sigma^2$  is hence denoted by $\mathcal{M}^2_{\sigma^2}$.
%;and $\mathcal{M}^2_{\mu_0, \sigma^2}$ those in $\mathcal{M}^2_{\sigma^2}$ with mean $\mu_0$, viz., $\mathcal{M}^2_{\mu_0, \sigma^2} = \{ P \in \mathcal{M}^2_{\sigma^2}: \mu(P) = \mu_0 \}$.

Recall that the total variation (TV) distance $\tv$ is defined via $\tv(P, Q) = \sup_{A\in \mathcal{B}(\mathbb R)}|P(A) - Q(A)|$ and is a metric on $\mathcal{M}$. Further, it is an \emph{integral probability metric} in the sense that for any pair of real numbers $c_1 < c_2$,
\begin{equation}\label{eqn:tv-integral}
   \tv(P, Q) = \frac{1}{c_2 - c_1}\sup_{c_1 \le f \le c_2} \left| \Expw_{X\sim P} f(X) - \Expw_{X \sim Q} f(X) \right|,
\end{equation}
where the supremum is taken over all measurable functions from $\mathbb R$ to $[c_1, c_2]$.
For any $P \in \mathcal{M}$, we denote by $\tvball(P, \varepsilon)$ the \emph{closed TV ball} of radius $\varepsilon$ around $P$,
\begin{equation}
    \tvball(P, \varepsilon) = \{ Q \in \mathcal{M} : \tv(P, Q) \le \varepsilon \}.
\end{equation}

We are now ready to define robust CIs and CSs (applicable beyond the scope of this paper).

\subsection{Definition of a Robust CS}

\begin{definition}[Robust CI and CS]\label{def:rci-rcs} Let $\mathcal{P} \subseteq \mathcal{M}$ and $\chi : \mathcal{P} \to \mathbb R$ be a functional. A sequence of measurable interval-valued functions $\{ \CI_t \}$
\begin{equation}
    \CI_t : \mathbb R^t \to \Bigl\{ [l, u] : -\infty \le l \le u \le \infty   \Bigr\}
\end{equation}
is called a sequence of $\varepsilon$-robust $(1-\alpha)$-confidence intervals over $\mathcal{P}$ for the functional $\chi$, or $(\varepsilon, 1-\alpha)$-RCIs for $(\mathcal{P}, \chi)$ for short, if
\begin{equation}
    \forall P \in \mathcal{P}, \ \forall Q \in\tvball(P, \varepsilon), \ \forall t \in \mathbb N^+ , \ \Prw_{X_i \iid Q} \left [ \chi(P) \in \CI_t (X_1, \dots, X_t) \right] \ge 1- \alpha.
\end{equation}
$\{ \CI_t \}$ is called an $\varepsilon$-robust $(1-\alpha)$-confidence sequence over $\mathcal{P}$ for $\chi$, or an $(\varepsilon, 1-\alpha)$-RCS for $(\mathcal{P}, \chi)$, if
\begin{equation}
    \forall P \in \mathcal{P}, \ \forall Q \in \tvball(P, \varepsilon), \ \Prw_{X_i \iid Q} \left [ \forall t \in \mathbb N^+ , \  \chi(P) \in \CI_t (X_1, \dots, X_t) \right] \ge 1- \alpha.
\end{equation}
\end{definition}
The central question that the current paper seeks to answer is the possibility of constructing a tight RCS for $(\mathcal{M}^2_{\sigma^2}, \mu)$. While our method actually produces RCS for $(\mathcal{M}^p_{\kappa}, \mu)$ where $p$ may be any real number $>1$, we shall focus primarily on the $p=2$ case for simplicity until \cref{sec:inf-var}.
\begin{remark}
% \normalfont
Our formulation of robustness via the TV-ball allows for a more general class of contaminations compared to the notion of ``gross error model'' by \cite{huber1964robust,huber1965robust,huber1968robust,huber1973robust} where data are drawn with $\varepsilon$ chance from an arbitrary other distribution. To wit, the $\varepsilon$-\emph{contamination neighborhood} in \citeauthor{huber1964robust}'s cited works, is
\begin{equation}\label{eqn:huber-contam}
    \contam(P, \varepsilon) := \{ (1-\varepsilon)P + \varepsilon P' : P' \in \mathcal{M} \},
\end{equation}
and it is easy to see that $\contam(P, \varepsilon) \subseteq \tvball(P,  \varepsilon)$.
\end{remark}

\begin{remark}\label{rmk:trivial-rcs}
Trivially, any RCS yields an RCI at a fixed time; and with any RCI, for example the one by \cite{lugosi2021robust}, one may define an RCS by simply applying a union bound to it (e.g.\ defining $\CI_t$ to be the $\left(\varepsilon, 1-\frac{\alpha}{T}\right)$-RCI for $t \le T$, and repeating $\CI_T$ for $t > T$, or alternatively defining it to be the entire real line before time $T$ and repeating $\CI_T$ after $T$). Our supermartingale techniques lead to much better performance than such trivial constructions, as shown at the end of \cref{sec:trim}.
\end{remark}

% However, as we shall present in \cref{fig:vs-lm}, their fixed-time $(1-\alpha)$-RCI is already very loose, and this trivial approach will lead to even poorer results. (Union bounds can be tight when the underlying events are nearly independent; here the miscoverage events at subsequent times are highly dependent.) We significantly improve upon it using supermartingale techniques.

\begin{remark}
We note that in this paper $\sigma,\epsilon$ are assumed known. As discussed in~\cite{wang2022catoni} which deals with the contamination-free case, not knowing a bound on $\sigma$ makes nonasymptotic inference impossible. 
Indeed, the class $\mathcal{M}^2$ (or any $\mathcal{M}^p$) is one that satisfies the impossibility result of \citet[Theorem 1]{bahadur1956nonexistence}. Hence by \citet[Corollary 2]{bahadur1956nonexistence}, no nontrivial CI exists for $\mu$ over these classes. The smaller class $\mathcal{M}^2_{\sigma^2}$ (or any $\mathcal{M}^p_{\kappa}$) with known bounds $\sigma^2$ or $\kappa$ suffers from no such limitation.  \cite{catoni2012challenging} and \cite{bhatt2022nearly} use the well-known ``Lepskii's method'' to obtain \emph{point estimators} that adapt to unknown variance (whose dominant terms of the mean squared error depend on the true variance which is smaller than $\sigma^2$), but it is important to note that no adaptive CI exists (or can exist).
%Foreknowledge on the confidence parameter $\alpha$ is also necessary (see e.g., \citet[Remark to Theorem 1]{lugosi2021robust}). 
However, we do not know if robust CIs require the contamination parameter $\varepsilon$ to be known, and no prior work, to the best of our knowledge, either bypasses this requirement or proves it necessary. While \citet[Section 4.1]{minsker2021robust} have a robust mean estimator that ``adapts'' to $\varepsilon$, the corresponding CI does not (the coverage probability depends on the unknown number of outliers). We leave the problem to future work. 
%so the restriction of knowing a bound on the variance of the distribution is not fully removable. Our goal is to produce a computationally simple and statistically efficient RCS for $\mu$. We do not know if the practicality of our methods can be maintained when adapting to $\epsilon$; we leave this as future work. 
\end{remark}

\subsection{Width Lower Bound}
It is well-known that robust mean estimators cannot attain consistency (meaning the width of robust CIs cannot shrink to zero), and various information-theoretic lower bounds state that an error scaling as a function (depending on the distribution class) of $\varepsilon$ is unavoidable for $\varepsilon$-robust methods \citep{chen2018robust,lugosi2021robust}.
%\footnote{Whether it is $\varepsilon$ or $\sqrt{\varepsilon}$ depends on the distribution class. Generally, the $\varepsilon$ lower bound holds for the sub-Gaussian class, while the more stringent $\sqrt{\varepsilon}$ lower bound holds for our finite variance class $\mathcal{M}^2_{\sigma^2}$. See \citet[Theorem 1.1 and 1.2]{cheng2019high}.} 
For the sake of completeness, we state (and prove in \proofs) the following lower bound for the diameter of RCIs and RCSs for $(\mathcal{M}^2_{\sigma^2},\mu)$.

\begin{lemma}[Lower bound for RCIs and RCSs for $(\mathcal{M}^2_{\sigma^2},\mu)$]\label{lem:lb}
Suppose $\{ \CI_t \}$ is a sequence of $(\varepsilon, 1-\alpha)$-RCIs for $(\mathcal{M}^2_{\sigma^2},\mu)$. Then, there is some $P \in \mathcal{M}^2_{\sigma^2}$ such that
\begin{equation}
   \forall t \in \mathbb N^+ , \ \Prw_{X_i \iid P} \left[  \diam(\CI_t ) \ge \sigma \sqrt{\varepsilon} \right] \ge 1- 2 \alpha.
\end{equation}
Further, if $\{ \CI_t \}$ is an $(\varepsilon, 1-\alpha)$-RCS for $(\mathcal{M}^2_{\sigma^2},\mu)$, then
\begin{equation}
    \Prw_{X_i \iid P} \left[ \forall t \in \mathbb N^+ , \ \diam(\CI_t ) \ge \sigma \sqrt{\varepsilon} \right] \ge 1- 2 \alpha.
\end{equation}
\end{lemma}

\subsection{Supermartingales and Ville's inequality}
\label{sec:ville}

A stochastic process $\{ M_t \}_{t \in \mathbb N}$ adapted to the filtration $\{ \mathcal{F}_t \}_{t \in \mathbb N}$ is called a supermartingale if $\Exp[M_t |  \mathcal{F}_{t-1}] \le M_{t-1}$ for all $t \in \mathbb N^+$. Nonnegative supermartingales are widely used to construct CSs  \citep{howard2021time, waudby2020estimating, wang2022catoni} thanks to the following theorem by \cite{ville1939etude}.
\begin{lemma}[Ville's inequality]\label{lem:vil}
Let $\{ M_t \}_{t \in \mathbb N}$ be a nonnegative supermartingale and $\alpha \in (0,1)$. Then, with probability at least $1-\alpha$, $\sup_{t \in \mathbb N^+} M_t \le \alpha^{-1} \Exp[M_0]$.
\end{lemma}
We refer the reader to \cite{howard2020time} for a short, modern proof of \cref{lem:vil}.

\sectionallcap{Main Results}

We now present our RCS construction. It is based on the design of new robust  supermartingales, inspired by recent Catoni-style supermartingales~\citep{wang2022catoni}.

\subsection{Robust Nonnegative Supermartingales}

 Consider the ``narrowest possible'' influence function in \citet[Equation (2.3)]{catoni2012challenging}, given by
 \begin{figure}[h]
 \begin{minipage}{.55\textwidth}
       \begin{equation}\label{eqn:influence}
    \phi(x) = \begin{cases} \log 2, & x \ge 1, \\ -\log(1-x+x^2/2), & 0 \le x < 1, \\ \log(1 + x + x^2/2), & -1 \le x < 0, \\ -\log 2 , & x < -1. \end{cases}
    \end{equation}
\end{minipage}%
\begin{minipage}{0.5\textwidth} 
    %\vspace{.3in}
    \centerline{\includegraphics[width=0.7\textwidth]{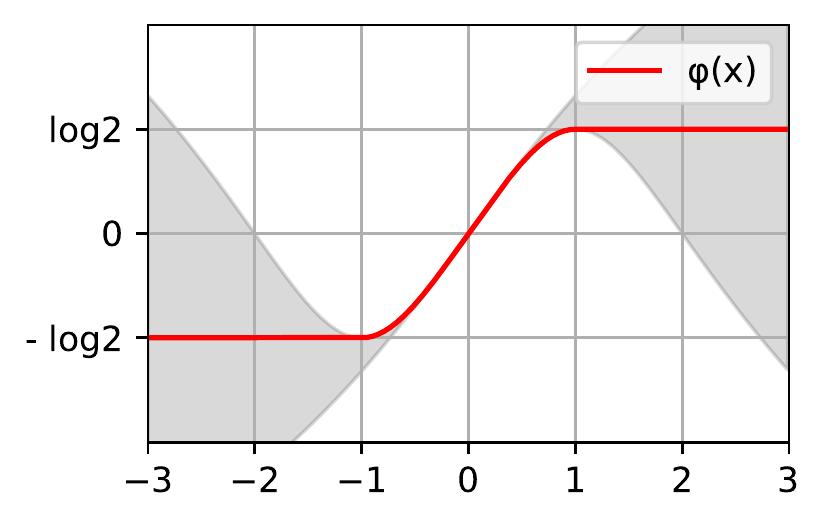}}
    \vspace{-.2in}
    \caption{\small The influence function~\eqref{eqn:influence}. The shaded region refers to Catoni's condition~\eqref{eqn:catoni-function-ineq}.}\label{fig:infl-func}
    %\vspace{-.1in}
\end{minipage}
\end{figure}

It is easy to check (see \cref{fig:infl-func}) that it satisfies
\begin{equation}\label{eqn:catoni-function-ineq}
    -\log(1-x+x^2/2) \le \phi(x) \le \log(1 + x + x^2/2)
\end{equation}
and $|\phi(x)| \le \log 2$.  \cite{catoni2012challenging} uses any function that satisfies~\eqref{eqn:catoni-function-ineq} to construct sub-Gaussian M-estimators of mean over $\mathcal{M}^2_{\sigma^2}$, explicitly stating that their objective ``should not be confused with robust statistics''. However, in a concurrent work, \cite{bhatt2022minimax} show that~\eqref{eqn:catoni-function-ineq} does lead to (fixed-time) robustness if boundedness is further assumed. We shall show that \emph{time-uniform} robustness is also viable by this idea of logarithmic influence function. The key is to obtain the following pair of ``robust nonnegative supermartingales'' over the  entire TV ball around $P$.

\begin{lemma}[Robust Catoni Supermartingales]\label{lem:rcatsm} For any $P \in \mathcal{M}^2_{\sigma^2}$ and $Q \in \tvball(P, \varepsilon)$, let
\begin{equation}
    X_1, X_2, \dots \iid  Q, \quad \text{each }X_t \text{ being } \mathcal{F}_t \text{-measurable}.
\end{equation}
Let $\{ \lambda_t \}_{t \in \mathbb N^+}$ be an $\{\mathcal{F}_t\}$-predictable process. Then, the following processes $\{ \Mrc_t \}_{t \in \mathbb N}$, $\{ \Nrc_t \}_{t \in \mathbb N}$, with $\Mrc_0 = \Nrc_0 = 1$, are nonnegative supermartingales:
\begin{gather}
    \Mrc_t = \prod_{i=1}^t \frac{ \exp \{ \phi(\lambda_i (X_i - \mu(P)))  \}}{ 1+ \lambda_i^2 \sigma^2 / 2 + 1.5 \varepsilon }, \label{eqn:mrc}
    \\
    \Nrc_t = \prod_{i=1}^t \frac{ \exp \{ -\phi(\lambda_i (X_i - \mu(P)))  \}}{  1 +\lambda_i^2 \sigma^2 / 2 + 1.5 \varepsilon }. \label{eqn:nrc}
\end{gather}
\end{lemma}
Note in particular that by setting $\epsilon=0$ and applying $1+x \le e^x$ to the denominator, one recovers the (non-robust) Catoni supermartingales of~\citet[Lemma 8]{wang2022catoni}.

The constant $1.5$ appearing above is not magical: it arises because the maximum and minimum values of $\exp(\phi(x))$ are 1/2 and 2, differing by 3/2. Indeed, the proof of \cref{lem:rcatsm}
%(in \proofs)
is based on the boundedness of $\exp(\phi(x))$, with the property~\eqref{eqn:tv-integral} of TV that leads to the translation of expectation from under $Q$ to under $P$. Given the centrality of the above lemma, and its uniqueness as being the first robust nonnegative supermartingale we are aware of, we present the proof immediately.

\begin{proof}[Proof of \cref{lem:rcatsm}] Since $|\phi(x)| \le \log 2$, we have $ 1/2 \le \exp \{ \phi(\lambda_t (x - \mu(P))) \} \le 2$. Note that
\begin{align}
    & \Exp \left[ \frac{ \exp \{ \phi(\lambda_t (X_t - \mu(P)))  \}}{ 1 + \lambda_t^2 \sigma^2 / 2 + 1.5 \varepsilon } \middle \vert \mathcal{F}_{t-1} \right]
    \\
     = & \frac{\Expw_{X_t \sim Q}\left[ \exp \{ \phi(\lambda_t (X_t - \mu(P))) \} \right]}{  1 +\lambda_t^2 \sigma^2 / 2 + 1.5 \varepsilon  }
    \\
      \le & \frac{ \Expw_{X_t \sim P }\left[ \exp \{ \phi(\lambda_t (X_t - \mu(P))) \}  \right] + 1.5 \varepsilon  }{  1 +\lambda_t^2 \sigma^2 / 2 + 1.5 \varepsilon   }
    \\
    \le & \frac{\Expw_{X_t \sim P }\left[  1 + \lambda_t (X_t - \mu(P)) + \lambda_t^2 (X_t - \mu(P))^2/2  \right] + 1.5\varepsilon  }{ 1 +\lambda_t^2 \sigma^2 / 2 + 1.5 \varepsilon   }
    \\
    = & \frac{ \Expw_{X_t \sim P }\left[  1 + \lambda_t^2 (X_t - \mu(P))^2/2  \right] + 1.5 \varepsilon  }{  1 +\lambda_t^2 \sigma^2 / 2 + 1.5 \varepsilon }
    \\
    = & \frac{ (1 + \lambda_t^2 v_2(P)/2) + 1.5 \varepsilon  }{ 1 +\lambda_t^2 \sigma^2 / 2 + 1.5 \varepsilon  }  \le 1.
\end{align}
The first inequality above is due to~\eqref{eqn:tv-integral} and $Q \in \tvball(P, \varepsilon)$. Hence $\{ \Mrc_t \}$ is a supermartingale. The proof for $\{ \Nrc_t \}$ is analogous.
\end{proof}

It is clear from the proof above that the i.i.d.\ assumption $X_i \sim Q \in \tvball(P,\varepsilon)$ can be relaxed; it suffices to assume that each $X_i$ is generated from \emph{some} distribution (or even a random distribution measurable w.r.t.\ $\mathcal{F}_{t-1}$) in $\tvball(P,\varepsilon)$. We keep the i.i.d.\ assumption (and notation) throughout the paper for simplicity.

\subsection{Our Huber-Robust CS}

Applying Ville's inequality (\cref{lem:vil}) on~\eqref{eqn:mrc} and~\eqref{eqn:nrc} leads to our robust confidence sequence, stated next and proved soon after.

\begin{theorem}[Huber-robust CS]\label{thm:rcs}
Define $f_t(m) := \sum_{i=1}^t \phi( \lambda_i (X_i - m))$. We define $\CIRC_t (X_1, \dots, X_t)$ to be
\begin{equation}\label{eqn:rcs}
    \Bigl\{ m \in \mathbb R : |f_t(m)| \le \log(2 / \alpha ) + \sum_{i=1}^t \log \left( 1 + \lambda_i^2 \sigma^2 / 2 + 1.5 \varepsilon \right) \Bigr\}.
\end{equation}
Then, $\{ \CIRC_t \}$ is an $(\varepsilon, 1-\alpha)$-RCS for $(\mathcal{M}^2_{\sigma^2},\mu)$. 
% \emph{(To do: point estimate)}
\end{theorem}

The confidence sequence is defined like M-estimators are; indeed, both end points of the interval $\CIRC_t$ are level points of $f_t(m)$ (where the inequality holds with equality). If a point estimate of the mean is desired, we may define $\rhatmu_t$ as the solution to the estimating equation $f_t(m)=0$. While it is always true that $\rhatmu_t \in \CIRC_t$, note that $\rhatmu_t$ is not the center of $\CIRC_t$ (despite its apparently symmetric definition). 
%Further, we observe that~\eqref{eqn:rcs} implicitly requires that $\varepsilon < 2/3$, as otherwise $\sum_{i=1}^t \log(1+\lambda_i^2 \sigma^2 + 1.5 \varepsilon)$ itself would surpass $t\log 2$, the maximum of $|f_t(m)|$, leading to a degenerate interval which is the entire real line.

% We prove \cref{thm:rcs} in \proofs. 

\begin{proof}[Proof of \cref{thm:rcs}] For any $Q \in \tvball(P, \varepsilon)$, applying Ville's inequality (\cref{lem:vil}) to the nonnegative supermartingale $\{ \Mrc_t \}$ under $Q$, we infer the following: with probability at least $1-\alpha/2$, we have that $ \forall t \in \mathbb N^+$, 
\begin{equation}
    f_t(\mu(P)) \le \sum_{i=1}^t \log\{ 1+ \lambda_i^2\sigma^2 / 2 +1.5\varepsilon \} + \log(2/\alpha). 
\end{equation}
Applying the same logic to $\{ \Nrc_t \}$:  with probability at least $1-\alpha/2$, we have that $ \forall t \in \mathbb N^+$, 
\begin{equation}
    -f_t(\mu(P)) \le \sum_{i=1}^t \log\{ 1+ \lambda_i^2\sigma^2 / 2 +1.5\varepsilon \} + \log(2/\alpha).
\end{equation}
Combining them with a union bound concludes the proof of the theorem.
\end{proof}

The confidence sequence, as many previous ones in the literature, is tunable up to a sequence of parameters $\{ \lambda_t \}$, which could be seen as weights put on the data sequence $\{ X_t \}$. Past works that involve such sequence parameters have oftentimes set it to be a decreasing sequence with a rate of decay approximately $t^{-1/2}$ \citep{waudby2020estimating,wang2022catoni}, which led to a $t^{-1/2} \polylog (t)$ rate of shrinkage in the width of the CS. However, as we have seen in \cref{lem:lb}, shrinkage of any RCIs (and thus any RCS) to zero at \emph{any} rate is impossible, so intuitive a decaying $\{ \lambda_t \}$ would no longer be preferred.
Indeed, as we shall soon see in \cref{sec:tune}, the CS would unboundedly inflate if $\{\lambda_t \}$ is set either to decrease, or increase, at a polynomial rate. Instead, a constant sequence $\lambda_1 = \lambda_2 = \cdots = \lambda$, where $\lambda \propto \varepsilon^{1/2} \sigma^{-1}$, leads to near-optimal width of the CS.

Under the choice $\lambda = 0.5 \varepsilon^{1/2} \sigma^{-1}$, we conduct a simulation with true distribution $P$ being $\mathcal{N}(0, 9)$, along with a $1/9$ chance of contamination from an asymmetric $0.75$-L\'evy stable distribution with location parameter 0 and skewness parameter $\beta = 0.5$. Our RCS with $\sigma^2 = 9$, $\varepsilon = 1/9$ is compared to a non-robust CS for $\sigma^2$-subGaussian distributions of \citet[Eq. (11)]{howard2021time} in \cref{fig:vs}. In practice, our RCS is much better than the concentration bound we will prove in \cref{thm:tight}, presented next.

We finally remark on the computational aspects of our RCS in \cref{thm:rcs}. As the function $f_t(m)$ is continuous and monotonic, root-finding algorithms including the bisection method, secant method, and Brent's method can efficiently calculate the upper and lower endpoints according to~\eqref{eqn:rcs}. We can further accelerate the root-finding via \emph{warm starting}, viz., initializing the iterate for finding $\max (\CIRC_t)$ with the previous solution $\max (\CIRC_{t-1})$.

\begin{figure}[!h]
  \begin{minipage}[c]{0.3\textwidth}
    \includegraphics[width=\textwidth]{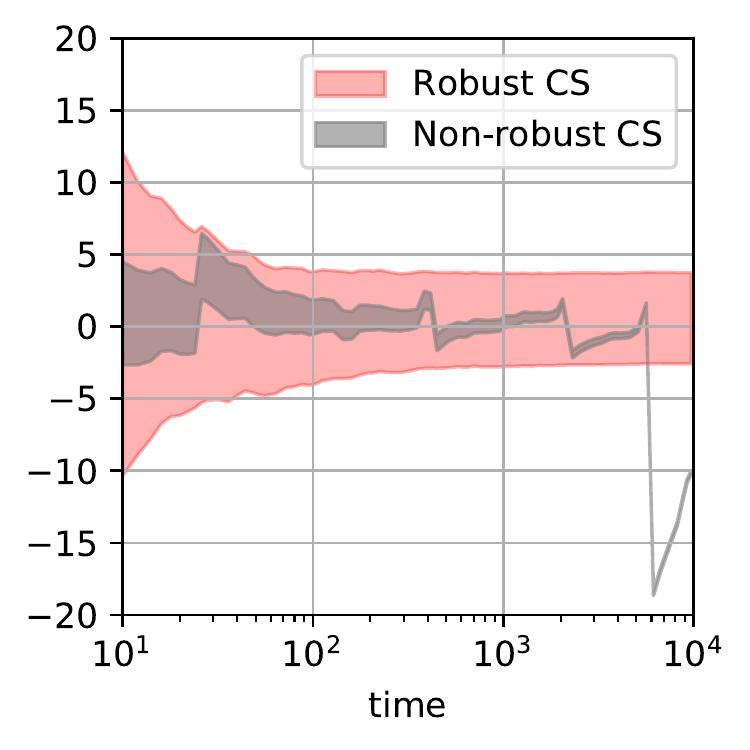}
  \end{minipage}\hfill
  \begin{minipage}[c]{0.7\textwidth}
    \caption{\small Our robust confidence sequence versus a non-robust confidence sequence by \cite{howard2021time}, under contaminated Gaussian data. Our RCS always cover the true mean $\mu(P)=0$ while non-robust CS does not. The $y$-axis of the plot scales one-to-one with the lower bound $\sigma \sqrt{\varepsilon} = 1$. Our RCS is tighter than our \cref{thm:tight} would predict.
    } \label{fig:vs}
  \end{minipage}
\end{figure}

\begin{figure*}[]
%\vspace{.3in}
\centerline{\includegraphics[width=\textwidth]{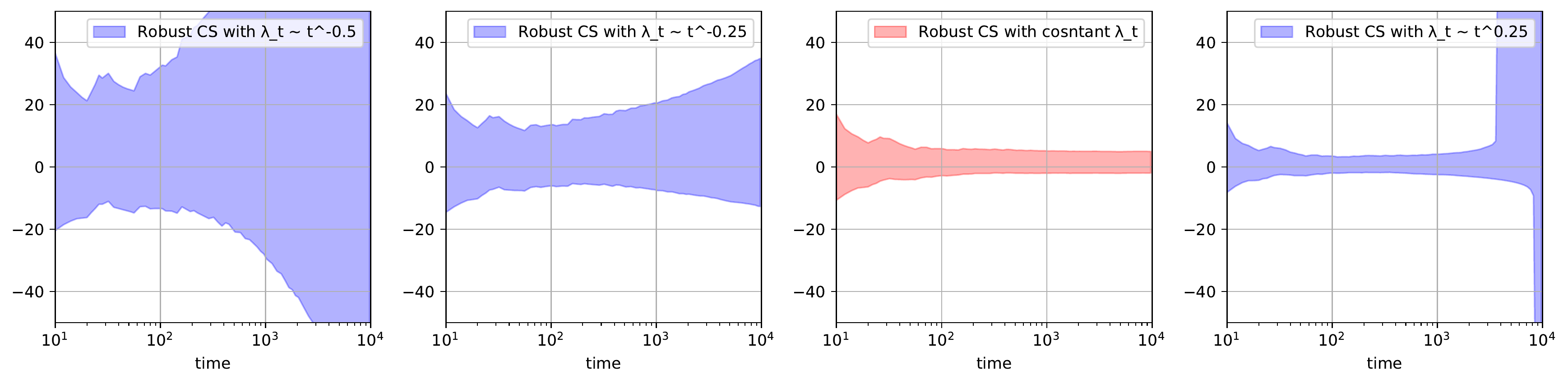}}
%\vspace{-.1in}
\caption{Our robust confidence sequences with different choice of $\{\lambda_t\}$. Indeed as discussed earlier, setting $\{\lambda_t \}$ to a constant multiple of $\sqrt\epsilon/\sigma$ (the third plot in red) controls the width, but other choices do not. 
%The $y$-axes of the plots scale one-to-one with the lower bound $\sigma \sqrt{\varepsilon} = 1$. 
%Our RCS is even tighter than our \cref{thm:tight} would predict.
}
\label{fig:lambda}
\end{figure*}

\subsection{Tuning and Width Analysis}
\label{sec:tune}

Our tightness bound for the RCS in \cref{thm:rcs} is stated in the form of a concentration bound, since the width of~\eqref{eqn:rcs}, unlike traditional confidence intervals, is random. Similar width concentration results have been used in random-width CSs by, for example, \cite{wang2022catoni}. Supposing that the weight sequence is taken constant, $\lambda_1  = \lambda_2 = \dots = \lambda$, we have the following bound.
\begin{theorem}\label{thm:tight} For any $0 < \delta, \alpha < 1$ and $0 < \varepsilon \le 1/7$,
suppose that $t > 4 \varepsilon^{-1} \log (4/\alpha \delta)$ and $\lambda^2 = \frac{\varepsilon}{4 \sigma^2}$. Under any corrupted distribution $Q \in  \tvball(P, \varepsilon)$, the $(\varepsilon,1-\alpha)$-RCS for $(\mathcal{M}^2_{\sigma^2},\mu)$ stated in \cref{thm:rcs} satisfies
\begin{equation}
    \Prw_{X_i \iid Q}\left[ \diam (\CIRC_t) \le 28 \sigma \sqrt{\varepsilon} \right] \ge 1 - \delta,
\end{equation}
matching the lower bound $\sigma\sqrt{\varepsilon}$.
\end{theorem}

Despite its centrality, the proof of \cref{thm:tight} is a bit long, and is thus deferred to \proofs.

We remark that \cref{thm:tight} answers the question of tuning in a satisfactory manner, since we have shown that setting $\lambda \propto \varepsilon^{1/2} \sigma^{-1}$ leads to an optimal width. We further demonstrate the point by following experiments: we set the true distribution $P$ to be $\mathcal{N}(0, 9)$, with $1/9$ chance of contamination from an asymmetric $0.3$-L\'evy stable distribution with location parameter $1000$ and skewness parameter $\beta = 0.5$. RCS with $\sigma^2 = 9$, $\varepsilon = 1/9$ and $\lambda_t = 0.5 \varepsilon^{1/2} \sigma^{-1} t^{u} $, where $u \in \{ -0.5, -0.25, 0, 0.25 \}$. The comparison is shown in \cref{fig:lambda}. We can observe from the plots that only setting $\{\lambda_t\}$ to a constant can keep the RCS within constant width.

% \begin{remark}\normalfont \cref{thm:tight} calls for knowing the robustness threshold $\varepsilon$ before
% employing our method, as the tightness constant $\lambda^2 = \frac{\varepsilon}{4 \sigma^2}$ depends on $\varepsilon$. 
% % We shall explain the necessity of such foresight in \cref{sec:sigma-eps}.
% %The necessity of such foresight is explained in \cite{lugosi2021robust}...
% \end{remark}

%The following plot is under $\sigma^2 = 1/\varepsilon = 40$ (so that the $y$-axis scales 1-to-1 with the lower bound ``unit'' width of $\sigma\sqrt{\varepsilon}$), with contamination from a symmetric 1.4-L\'evy-stable distribution at mean 500. The actual width is smaller than our theoretical value of $32$ units.

%\includegraphics[width=2.5in]{robust-catoni-cs/medias/robustcatoni.png}

%\subsection{Computational Aspects}

\sectionallcap{An Empirical Comparison}
% \subsection{Necessary Foreknowledge of $\sigma$ and $\varepsilon$}\label{sec:sigma-eps}

Robust uncertainty quantification (i.e.\ valid $(1-\alpha)$-RCIs or RCSs) has not been a key focus of past work, which has been dominated by point estimators. Thus, we must in some sense extract an RCI from related work as a point of comparison, and we do this below.

\subsection{Relationship between Robustness Models}\label{sec:rel}

Apart from the robustness model in \cref{def:rci-rcs} we consider (which is itself a generalization of the model of \cite{huber1964robust}), there is in the literature another model of robustness, where, instead of an $\varepsilon$ chance of random contamination to the distribution, there is an adversary who can \emph{replace} up to $\varepsilon$ fraction of the data \emph{after} they are generated from the true distribution. This fully adversarial set-up has been considered, among others, by \cite{diakonikolas2019robust, cheng2019high}, and \cite{lugosi2021robust} whose distribution assumption closely matches ours (both being $\mathcal{M}^2_{\sigma^2}$). To compare our results to those by \cite{lugosi2021robust}, we first formally relate the robustness model they operate on to ours.

Let $\varepsilon\in(0,1)$ and $t \in \mathbb N^+$. Define $\mathcal{R}_{\varepsilon, t}$ as the class of all functions $\mathbb R^t \to \mathbb R^t$ that change at most $\lfloor \varepsilon t \rfloor$ of the $t$
coordinates into constants. For example, the function $c(x, y, z) = (x, 100, z)$ satisfies $c \in \mathcal{R}_{\varepsilon, 3}$ for all $\varepsilon \ge 1/3$. We call $\mathcal{R}_{\varepsilon, t}$ the set of \emph{$\varepsilon$-replacements}.

\begin{definition}[Replacement robust CI (RRCI)]\label{def:repl} Let $\mathcal{P} \subseteq \mathcal{M}$ and $\chi : \mathcal{P} \to \mathbb R$ be a functional. A sequence of measurable interval-valued functions $\{ \CI_t \}$ is called a sequence of $\varepsilon$-replacement robust $(1-\alpha)$-confidence intervals over $\mathcal{P}$ for the functional $\chi$, or $(\varepsilon, 1-\alpha)$-RRCIs for $(\mathcal{P}, \chi)$, if
\begin{equation}
    \forall P \in \mathcal{P}, \ \forall t\in \mathbb N^+, \ \Prw_{X_i \iid P} \left [{\forall c \in \mathcal{R}_{\varepsilon, t}, \ } \chi(P) \in \CI_t \left(c(X_1, \dots, X_t)\right) \right] \ge 1- \alpha.
\end{equation}
\end{definition}
We emphasize that the probability bound is uniform over $\forall c \in \mathcal{R}_{\varepsilon, t}$ because the adversary may examine the data before deciding on how to corrupt them with a replacement $c \in \mathcal{R}_{\varepsilon, t}$. While it is tempting to sequentialize \cref{def:repl} by, say, putting ``$\forall t\in \mathbb N^+$'' inside the probability bound, such ostensible generalization does not make practical sense.
As we have pointed out, \cref{def:repl} is motivated by the setting where an adversary corrupts the data \emph{offline}, after they are entirely generated, which is an inherently fixed-time scenario. By contrast, sequential statistics is concerned with the scenario where we face an infinite stream of data, and inference is conducted \emph{online} along the way, instead of post-hoc.
This is why we hold ourselves back from defining ``RRCS''. Still, we define RRCIs largely due to its following relation with RCIs in \cref{def:rci-rcs}.

\begin{lemma}\label{lem:repl-obliv}
Any sequence of $(2\varepsilon, 1-\alpha)$-RRCIs for $(\mathcal{P},\chi)$ is a sequence of $(\varepsilon, 1 - (\alpha + \e^{-2t\varepsilon^2}))$-RCIs for $(\mathcal{P},\chi)$.
\end{lemma}
\cref{lem:repl-obliv} states that RRCIs can simulate RCIs up to a constant. While similar statements have been known, for example in \citet[Corollary 2.1]{jlilecturenotes}, we prove \cref{lem:repl-obliv} in full in \proofs. We now have the necessary terminology to introduce the trimmed mean RCI.

\subsection{Comparison with the Trimmed Mean RCI}
\label{sec:trim}

A variant of the trimmed mean based on sample splitting is shown by \cite{lugosi2021robust}
to be robustly concentrated in the replacement sense, which can be used to construct a sequence of RRCIs. We rephrase as follows the main finding of \cite{lugosi2021robust}.
\begin{theorem}[Theorem 1 of \cite{lugosi2021robust}, rephrased]\label{thm:lmci} Let $\widehat{\mu}_t(x_1, \dots, x_t)$ be (a slight variant of) the trimmed mean of $x_1, \dots, x_t$. Suppose $0 < \alpha < 1$, and $
    t/2 \ge \log(1/4\alpha)$.
Then, for all $P \in \mathcal{M}^2_{\sigma^2}$, with probability at least $1 - \alpha$ over $X_1, X_2, \iid P$,
\begin{equation}\label{eqn:lm-conc}
     |\widehat{\mu}_t(c(X_1, \dots, X_t)) - \mu(P)| \le 12\sqrt{2\varepsilon'}\sigma + 2\sigma \sqrt{\frac{\log(4/\alpha)}{t/2}}
\end{equation}
uniformly over all $c \in \mathcal{R}_{\varepsilon, t}$, where
\begin{equation}
    \varepsilon' = 8\varepsilon + \frac{12\log(4/\alpha)}{t/2}.
\end{equation}
\end{theorem}
In particular, if we are to construct $(\varepsilon, 1-\alpha)$-RRCIs from the concentration bound~\eqref{eqn:lm-conc}, the width is strictly and deterministically greater than $
    24\sqrt{2\varepsilon'} \sigma$ which is in turn greater than $96 \sqrt{\varepsilon}$.\footnote{In \citet[Theorem 1]{lugosi2021robust}, the first displayed equation is correct (which we quote in~\eqref{eqn:lm-conc}), while their second displayed equation is off by a constant.}
For example, if we assume 
\begin{equation}
    t \ge \max\left\{ \frac{\varepsilon ^{-1}\log(4/\alpha)}{0.09}, 2\log(1/4\alpha) \right\}.
\end{equation}
Then,
$[\widehat{\mu}_t \pm 49 \sqrt{\varepsilon}\sigma]$
is a valid RRCI.
Now recall \cref{lem:repl-obliv}. To simulate a sequence of  $(\varepsilon, 1-\alpha)$-RCIs, one needs
 a sequence of $(2\varepsilon, 1-\alpha')$-RRCIs where $\alpha'$ is slightly smaller than $\alpha$. Hence, the RCIs simulated by \cref{thm:lmci} are at least $96\sqrt{2 \varepsilon}$ wide.
 Compare this with our RC\emph{S} which is only $28\sqrt{\varepsilon}\sigma$ wide, with high probability (\cref{thm:tight}). 

As mentioned before, \cref{thm:tight} is quite conservative about the tightness of our RCS. In practice, the advantage of our approach is even more pronounced.
The comparison shown in \cref{fig:vs-lm} is conducted under $\sigma^2 = 1/\varepsilon = 36$, with contamination from an asymmetric 0.3-L\'evy stable distribution with location parameter $1000$ and skewness parameter $\beta = 0.5$. When $t$ is small, \citeauthor{lugosi2021robust}'s \citeyearpar{lugosi2021robust} split-sample trimming is even undefined (all of the data trimmed);
%; actually, as long as $\varepsilon \ge 1/16$, \cite{lugosi2021robust} always trim \emph{all} of the data.
and when their RCIs are defined, they are in a cosmic distance compared to our RCS. Other robust mean estimators over $\mathcal{M}_{\sigma^2}^2$ (e.g.\ \citealp{depersin2022robust}) suffer from even larger, ``galactic" concentration constants.

We remark that the advantage of our approach lies in the fact that it directly solves for the interval, instead of hinging on some concentration inequality on a point estimator (which causes many other works to have large constants, in turn causing very loose CIs). We do use concentration bounds to \emph{analyze} our widths in \cref{thm:tight}, but not to \emph{construct} them; thus constant-factor bottlenecks in theoretical analysis do hurt practical performance of other estimators, but not ours. Even if concentration constants are improved in other methods, they are unlikely to beat ours in practice as Ville’s inequality is known for its tightness.

\begin{wrapfigure}{r}{0.4\textwidth}
\centering
\includegraphics[width=0.4\textwidth]{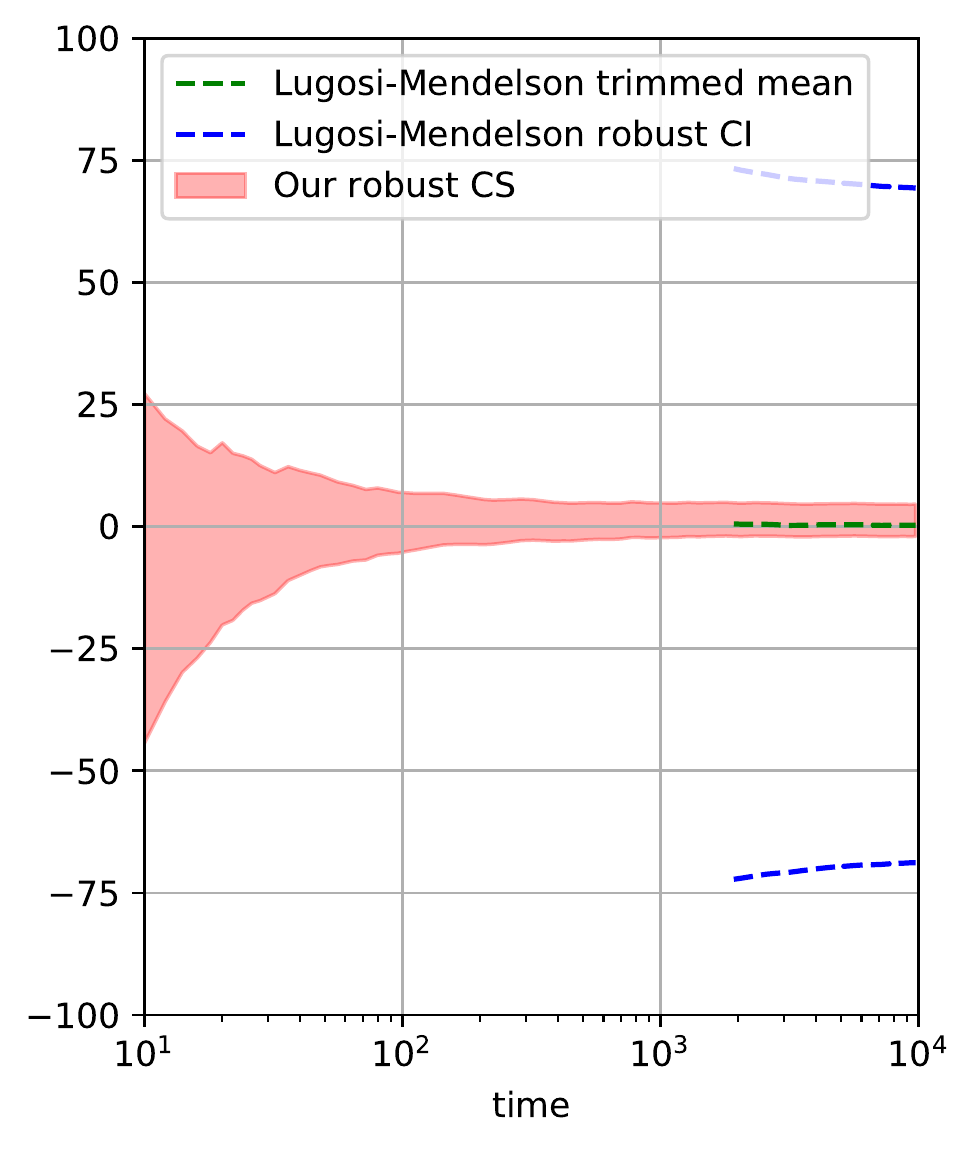}
\vspace{-20pt}
\caption{\small Comparison of our RCS with \citeauthor{lugosi2021robust}'s \citeyearpar{lugosi2021robust} RCIs. The lack of dashed lines before $t \approx 2 \times 10^3$ is because all the data are trimmed using the latter approach. Again, the $y$-axis of the plot scales one-to-one with the lower bound $\sigma \sqrt{\varepsilon} = 1$.} 
\label{fig:vs-lm}
%\vspace{-60pt}
\end{wrapfigure}

Another important advantage of our M-estimation approach lies in its large \emph{break-down point}. That is, it tolerates large amount of corruption. The advantage is also mentioned by \cite{bhatt2022minimax} in their fixed-time study, their robust M-estimator breaking down at a higher $\varepsilon$ ($36\%$ for $\mathcal{M}_{\sigma^2}^2$) compared to the trimmed mean of \cite{lugosi2021robust} (which breaks down at $\varepsilon = 1/16$). In our \cref{thm:rcs}, taking each $\lambda_i = \lambda$, the interval in~\eqref{eqn:rcs} will not span the entire $\mathbb R$ if $t \log 2 > \log(2/\alpha) + t \log (1 + \lambda^2 \sigma^2/2 + 1.5 \varepsilon)$, which will happen for large $t$ as long as $2 > 1 + \lambda^2 \sigma^2/2 + 1.5 \varepsilon$. Under our choice $\lambda^2 = \frac{\varepsilon}{4\sigma^2}$ this amounts to $\varepsilon < 8/13$.

Finally, we return to the trivial RCS construction mentioned in \cref{rmk:trivial-rcs}.  First, as we can observe in \cref{fig:vs-lm}, even the best fixed-time $(1-\alpha)$-RCIs are already very loose as they suffer from impractical large constants. Second, the use of union bounds will lead to even poorer results as union bounds can only be tight when the underlying events are nearly independent; here the miscoverage events at subsequent times are highly dependent. Nonnegative supermartingales and Ville's inequality, on the other hand, pay no such (indeed, constant-level) price. It is worth stressing here that \emph{constants matter} in confidence sequences, especially because they are widely used in deployed IT services.

\sectionallcap{Extensions}
\subsection{Robust Test Supermartingales}
The well-known duality between confidence intervals and hypothesis testing also extends to the robust sequential setting. For some set $S$ in the range of a functional $\chi$, consider the null hypothesis:
\begin{equation}
    \mathcal{H}_0 = \{ P \in \mathcal{P} : \chi(P) \in S \}.
\end{equation}
If $\{\CI_t\}$ is an $(\varepsilon, 1-\alpha)$-RCS for $(\mathcal{P}, \chi)$, then the sequential decision to reject $\mathcal{H}_0$ whenever $S$ and $\CI_t$ do not intersect attains the following robustified control of type I error:
\begin{equation}
    \forall Q \in  \tvball(\mathcal{H}_0, \varepsilon), \Prw_{X_i \iid Q} [\text{ever rejecting } \mathcal{H}_0] \le \alpha.
\end{equation}
Here $\tvball(\mathcal{H}_0, \varepsilon)$ is the closed $\varepsilon$-neighorhood of the null set $\mathcal{H}_0$ under the TV metric, i.e., $\bigcup_{P \in \mathcal{H}_0} \tvball(P, \varepsilon)$. Robust tests  thus  enlarge $\mathcal{H}_0$ to its neighbor $\tvball(\mathcal{H}_0, \varepsilon)$.

In particular, our RCS in \cref{thm:rcs} can be used to robustly test if $\mu(P) = \mu_0$. Equivalently, $\mu(P) = \mu_0$ is rejected when both of the supermartingales in \cref{lem:rcatsm} (replacing $\mu(P)$ with $\mu_0$) surpass $2/\alpha$. Moreover, similar to the discussion of \citet[Section 10.4]{wang2022catoni}, each one of the supermartingale pair in \cref{lem:rcatsm} works implicitly to test a one-sided composite null. To wit, whenever
\begin{gather}\label{eqn:mrc-mu0}
    \Mrc_t = \prod_{i=1}^t \frac{ \exp \{ \phi(\lambda_i (X_i - \mu_0))  \}}{  1+\lambda_i^2 \sigma^2 / 2 + 1.5 \varepsilon }
\end{gather}
is greater than $1/\alpha$, reject the one-sided composite
% \begin{multline}
    \[
    \mathcal{H}_0 
    % \mu^{-1}( (-\infty, \mu_0] ) \cap \mathcal{M}^2_{\sigma^2}
    % \\
    =  \{ P \in \mathcal{M}^2_{\sigma^2} : \mu(P) \le \mu_0 \}.
    \]
% \end{multline}
Then, the sequential type I error is controlled at level $\alpha$ over the enlarged null $\tvball(\mathcal{H}_0, \varepsilon)$. While it is natural to think that the powers of these tests, characterized by the growth of the test processes under the alternative, are reduced due to the trade-off with robustification (i.e.\ preventing false rejections due to corrupted data), an exponential rate of growth, similar to the nonrobust tests, can still be achieved under sufficient separation between the null and the alternative. We formalize this with the following dual to \cref{thm:tight}, proved in \proofs:
\begin{corollary}\label{cor:growth} Under the assumptions of \cref{thm:tight} except $t \ge 4\varepsilon^{-1} \log(4/\alpha\delta)$,
when $\mu(P) > \mu_0 + 14\sigma \sqrt{\varepsilon}$, the process $\{ \Mrc_t \}$ in~\eqref{eqn:mrc-mu0} %\footnote{Note that~\eqref{eqn:mrc-mu0} slightly differs from~\eqref{eqn:mrc} as $\mu_0$ is in the place of $\mu(P)$ to formulate testing.}
grows exponentially with
\begin{equation}
 \Prw_{X_i \iid Q} \left[  \Mrc_t > \frac{\delta}{4} \exp\left(\frac{t \varepsilon}{4}\right) \right] \ge 1- \delta/2.
\end{equation}
\end{corollary}

\subsection{Infinite Variance}\label{sec:inf-var}
Recalling the notation around~\eqref{eq:defn-v-p}, we now turn to the question of constructing RCS for $(\mathcal{M}^p_{\kappa}, \mu)$, for $1 < p < 2$.

Now, instead of~\eqref{eqn:influence}, we define
\begin{equation}\label{eqn:influence-p}
    \phi_p(x) = \begin{cases} \log p, & x \ge 1, \\ -\log(1-x+x^p/p), & 0 \le x < 1, \\ \log(1 + x + |x|^p/p), & -1 \le x < 0, \\ -\log p , & x < -1. \end{cases}
\end{equation}
Then, akin to \cref{lem:rcatsm}, we have,
\begin{lemma}\label{lem:rcatsm-infvar} For any $P \in \mathcal{M}^p_{\kappa}$ and $Q \in \tvball(P, \varepsilon)$, let $
    X_1, X_2, \dots \iid  Q$, each $X_t$ being $\mathcal{F}_t$-measurable.
Let $\{ \lambda_t \}_{t \in \mathbb N^+}$ be a $\{\mathcal{F}_t\}$-predictable process. Then, the following processes $\{ {\Mrc_t}^p \}_{t \in \mathbb N}$, $\{ {\Nrc_t}^p \}_{t \in \mathbb N}$ are nonnegative supermartingales adapted to $\{\mathcal{F}_t\}$:
\begin{gather}
{\Mrc_0}^p = {\Nrc_0}^p = 1,  
    \\
    {\Mrc_t}^p = \prod_{i=1}^t \frac{ \exp \{ \phi_p(\lambda_i (X_i - \mu(P)))  \}}{ 1+ \lambda_i^p \kappa /p + (p-1/p) \varepsilon }, \label{eqn:mrc-p}
    \\
    {\Nrc_t}^p = \prod_{i=1}^t \frac{ \exp \{ -\phi_p(\lambda_i (X_i - \mu(P)))  \}}{  1 +\lambda_i^p \kappa / p + (p-1/p) \varepsilon }.\label{eqn:nrc-p}
\end{gather}
\end{lemma}
Applying Ville's inequality, we can extend \cref{thm:rcs}.
\begin{theorem}\label{thm:rcs-p}
Define $f_{pt}(m) := \sum_{i=1}^t \phi_p( \lambda_i (X_i - m))$. We define ${\CIRC_t}^p (X_1, \dots, X_t)$ to be
\begin{equation}\label{eqn:rcs-p}
  \Bigl\{ m \in \mathbb R : |f_{pt}(m)| \le  \log(2 / \alpha ) + \sum_{i=1}^t \log \left\{ 1 + \lambda_i^p \kappa/ p + (p-1/p) \varepsilon \right\} \Bigr\},
\end{equation}
Then, $\{ {\CIRC_t}^p \}$ is an $(\varepsilon, 1-\alpha)$-RCS for $(\mathcal{M}^p_{\kappa},\mu)$. 
% \emph{(To do: point estimate)}
\end{theorem} 
This RCS, again, has random widths and satisfies a high probability width bound:
\begin{theorem}\label{thm:tight-infvar} For any $0<\alpha, \delta < 1$ and $0 < \varepsilon \le \frac{p-1}{7p}$, suppose $t \ge \varepsilon^{-1} \log(4/\alpha \delta)$, under any corrupted distribution $Q \in  \tvball(P, \varepsilon)$, the $( \varepsilon, 1-\alpha)$-RCS $\{ {\CIRC_t}^p \}$ with $\lambda_t = (\varepsilon / \kappa)^{1/p}$ for all $t$ satisfies
\begin{equation}\label{eqn:inf-var-bound}
   \Prw_{X_i \iid Q}\left[ \diam({\CIRC_t}^p) \le  \frac{14p}{p-1} \kappa^{1/p} \varepsilon^{(p-1)/p} \right] \ge 1- \delta.
\end{equation}
\end{theorem}

% Since the proofs of the above claims follow a similar line of argument as the presented proofs of our main lemmas and theorems, 

We defer the above proofs to \proofs. The lower bound from \citet[Theorem 4.4]{bhatt2022minimax} implies that RCIs for $(\mathcal{M}^p_{\kappa}, \mu)$ must, with high probability, have widths at least $\frac{1}{8} \kappa^{1/p} \varepsilon^{(p-1)/p}$. Thus, our width is  optimal up to constant factors.

We remark that in~\eqref{eqn:influence-p}, we set the coefficient of $x^p$ and $|x|^p$ to be $1/p$, as this leads to a succinct, aesthetically pleasing bound~\eqref{eqn:inf-var-bound}. One may also follow the tuning technique by \cite{bhatt2022nearly} in their fixed-time, non-robust setting, to set the coefficient unknown and then optimize over it. This may lead to constant-level improvement.

\subsection{Robust Betting}\label{sec:bet}

We finally demonstrate a sibling case of potential interest, the robust extension of the Kelly betting scheme by \cite{waudby2020estimating} for bounded data. Let $\mathcal{M}_{[0,1]}$ be the set of all distributions on $[0,1]$; the original and corrupted distribution both belong to this class. The following theorem (proved in \proofs) establishes a robust supermartingale analogous to the capital process when betting on $\mu(P)$ (cf.\ \citet[Section 4]{waudby2020estimating}).
\begin{theorem}\label{thm:bet}
    Let $P \in \mathcal{M}_{[0,1]}$ and $Q \in \tvball(P, \varepsilon) \cap \mathcal{M}_{[0,1]}$. Suppose $X_1, X_2,\dots \iid Q$ and each $X_t$ is $\mathcal{F}_t$-measurable. The process $\{ L_t \}$ defined as follows is a supermartingale:
%\begin{equation}\label{eqn:bet-proc}
 %   K_t = \prod_{i=1}^t \left(1-\frac{2\varepsilon}{2\varepsilon + 1} + \lambda_i(X_i - \mu(P))\right)
%\end{equation}
\begin{gather} 
     %K_t = \prod_{i=1}^t \left(1-\frac{2\varepsilon}{2\varepsilon + 1} + \lambda_i(X_i - \mu(P))\right), \ |\lambda_i| \le  \frac{1}{2\varepsilon + 1}; \label{eqn:bet-proc}
     %\\ 
     L_t = \prod_{i=1}^t ( 1 + \lambda_i(X_i  - \mu(P)) -\varepsilon|\lambda_i| ), \ |\lambda_i|\le \frac{1}{1+\varepsilon}. \label{eqn:bet-proc}
\end{gather}
\end{theorem}
%Both of the two supermartingales above, like \eqref{eqn:mrc}, have an $\varepsilon$-dependent term, serving as the trade-off between rate of growth and safeguard against falsely rejecting the null due to corrupted observations when testing $\mu(P) = \mu_0$ (replacing $\mu(P)$ with $\mu_0$). For \eqref{eqn:bet-proc}, the trade-off is a $2\varepsilon/(1+2\varepsilon)$ factor of discount that is in effect even if the ``bet" $\lambda_i$ is set to 0, much like \eqref{eqn:mrc}. The process \eqref{eqn:bet-proc-2} behaves differently as its trade-off is smaller if the ``bet" $\lambda_i$ is smaller. We leave it as future work to evaluate these two modes of robustification.

Unlike \eqref{eqn:mrc}, the $\varepsilon$-dependent term here (which is like an ``insurance cost'' against corruptions) scales linearly with $\lambda_i$, and does not discount the process if $\lambda_i$ is 0. When testing $\mu(P) = \mu_0$ (replacing $\mu(P)$ with $\mu_0$ in \eqref{eqn:bet-proc}), the trade-off between rate of growth and safeguard against falsely rejecting the null due to corrupted observations is smaller if the ``bet" $\lambda_i$ is smaller. 

\sectionallcap{Conclusion}

In this paper, we derive Huber-robust confidence sequences for the class of distributions with bounded $p$-th central moments, the first Huber-robust CSs we are aware of for any class. These are based on the design of new robust nonnegative supermartingales. 

Our CS matches the width lower bound (up to a constant) and it performs even better than robust nonsequential  (fixed-time) CIs in the literature. As referenced earlier, our methods will enable immediate robustification of downstream applications of confidence sequences, e.g., multi-armed bandits with contaminated distributions, and A/B/n testing within existing experimentation pipelines in multiple companies in the IT industry. 

We hope that followup work can extend our ideas to multidimensional settings, perhaps utilizing the recent reductions presented in~\cite{prasad2020robust} (``A robust univariate mean estimator is all you need''). This does not appear to be straightforward because the aforementioned paper often hides constants (which do not matter for the rate-optimality results they are interested in, but which do matter for achieving the desired level $\alpha$), and it operates in a fixed-sample size setting. It appears that many theoretical and practical details need to be worked out for an exact and practical $(1-\alpha)$-RCIs or RCS to be feasible.

Other problems of potential future interest that this paper brings up include the necessity (or lack thereof) of foreknowledge of the parameter $\varepsilon$,
%as well as a width analysis in the infinite variance case. 
as well as a further understanding of the betting scheme and especially its $\lambda_i$-proportional robustification cost in \cref{sec:bet}.
Last, given the well known connections between robust statistics and differential privacy~\citep{dwork2009differential}, the privacy implications of this work may be interesting to pursue. Indeed, large-scale private and robust sequential experimentation is certainly of increasing interest in the IT industry.

\bibliography{main.bib}

\begin{thebibliography}{49}
\providecommand{\natexlab}[1]{#1}
\providecommand{\url}[1]{\texttt{#1}}
\expandafter\ifx\csname urlstyle\endcsname\relax
  \providecommand{\doi}[1]{doi: #1}\else
  \providecommand{\doi}{doi: \begingroup \urlstyle{rm}\Url}\fi

\bibitem[{Analytics for Target}()]{AdobeExperimentation}
{Analytics for Target}.
\newblock Experimentation panel.
\newblock URL
  \url{https://experienceleague.adobe.com/docs/analytics-platform/using/cja-workspace/panels/experimentation.html?lang=en}.
\newblock Accessed: 13 Oct, 2022.

\bibitem[Anscombe(1960)]{anscombe1960rejection}
F.~J. Anscombe.
\newblock Rejection of outliers.
\newblock \emph{Technometrics}, 2\penalty0 (2):\penalty0 123--146, 1960.

\bibitem[Bahadur and Savage(1956)]{bahadur1956nonexistence}
R.~R. Bahadur and L.~J. Savage.
\newblock The nonexistence of certain statistical procedures in nonparametric
  problems.
\newblock \emph{The Annals of Mathematical Statistics}, 27\penalty0
  (4):\penalty0 1115--1122, 1956.

\bibitem[Bhatt et~al.(2022{\natexlab{a}})Bhatt, Fang, Li, and
  Samorodnitsky]{bhatt2022minimax}
S.~Bhatt, G.~Fang, P.~Li, and G.~Samorodnitsky.
\newblock Minimax {M}-estimation under adversarial contamination.
\newblock In \emph{International Conference on Machine Learning}, pages
  1906--1924. PMLR, 2022{\natexlab{a}}.

\bibitem[Bhatt et~al.(2022{\natexlab{b}})Bhatt, Fang, Li, and
  Samorodnitsky]{bhatt2022nearly}
S.~Bhatt, G.~Fang, P.~Li, and G.~Samorodnitsky.
\newblock Nearly optimal {C}atoni’s {M}-estimator for infinite variance.
\newblock In \emph{International Conference on Machine Learning}, pages
  1925--1944. PMLR, 2022{\natexlab{b}}.

\bibitem[Bickel(1965)]{bickel1965some}
P.~J. Bickel.
\newblock On some robust estimates of location.
\newblock \emph{The Annals of Mathematical Statistics}, pages 847--858, 1965.

\bibitem[Box(1953)]{box1953non}
G.~E. Box.
\newblock Non-normality and tests on variances.
\newblock \emph{Biometrika}, 40\penalty0 (3/4):\penalty0 318--335, 1953.

\bibitem[Catoni(2012)]{catoni2012challenging}
O.~Catoni.
\newblock Challenging the empirical mean and empirical variance: a deviation
  study.
\newblock In \emph{Annales de l'IHP Probabilit{\'e}s et Statistiques},
  volume~48, pages 1148--1185, 2012.

\bibitem[Chen et~al.(2018)Chen, Gao, and Ren]{chen2018robust}
M.~Chen, C.~Gao, and Z.~Ren.
\newblock Robust covariance and scatter matrix estimation under {H}uber’s
  contamination model.
\newblock \emph{The Annals of Statistics}, 46\penalty0 (5):\penalty0
  1932--1960, 2018.

\bibitem[Cheng et~al.(2019)Cheng, Diakonikolas, and Ge]{cheng2019high}
Y.~Cheng, I.~Diakonikolas, and R.~Ge.
\newblock High-dimensional robust mean estimation in nearly-linear time.
\newblock In \emph{Proceedings of the Thirtieth Annual ACM-SIAM Symposium on
  Discrete Algorithms}, pages 2755--2771. SIAM, 2019.

\bibitem[Darling and Robbins(1967)]{darling1967confidence}
D.~A. Darling and H.~Robbins.
\newblock Confidence sequences for mean, variance, and median.
\newblock \emph{Proceedings of the National Academy of Sciences of the United
  States of America}, 58\penalty0 (1):\penalty0 66, 1967.

\bibitem[Depersin and Lecu{\'e}(2022)]{depersin2022robust}
J.~Depersin and G.~Lecu{\'e}.
\newblock Robust sub-{G}aussian estimation of a mean vector in nearly linear
  time.
\newblock \emph{The Annals of Statistics}, 50\penalty0 (1):\penalty0 511--536,
  2022.

\bibitem[Diakonikolas et~al.(2019)Diakonikolas, Kamath, Kane, Li, Moitra, and
  Stewart]{diakonikolas2019robust}
I.~Diakonikolas, G.~Kamath, D.~Kane, J.~Li, A.~Moitra, and A.~Stewart.
\newblock Robust estimators in high-dimensions without the computational
  intractability.
\newblock \emph{SIAM Journal on Computing}, 48\penalty0 (2):\penalty0 742--864,
  2019.

\bibitem[Dwork and Lei(2009)]{dwork2009differential}
C.~Dwork and J.~Lei.
\newblock Differential privacy and robust statistics.
\newblock In \emph{Proceedings of the forty-first annual ACM symposium on
  Theory of computing}, pages 371--380, 2009.

\bibitem[Evidently()]{EvidentlyAmazon}
Evidently.
\newblock How evidently calculates results.
\newblock URL
  \url{https://docs.aws.amazon.com/AmazonCloudWatch/latest/monitoring/CloudWatch-Evidently-calculate-results.html}.
\newblock Accessed: 13 Oct, 2022.

\bibitem[Gibbs and Su(2002)]{gibbs2002choosing}
A.~L. Gibbs and F.~E. Su.
\newblock On choosing and bounding probability metrics.
\newblock \emph{International statistical review}, 70\penalty0 (3):\penalty0
  419--435, 2002.

\bibitem[Gr{\"u}nwald et~al.(2019)Gr{\"u}nwald, de~Heide, and
  Koolen]{grunwald2020safe}
P.~Gr{\"u}nwald, R.~de~Heide, and W.~M. Koolen.
\newblock Safe testing.
\newblock \emph{arXiv:1906.07801}, 2019.

\bibitem[Hampel(1968)]{hampel1968contributions}
F.~R. Hampel.
\newblock \emph{Contributions to the theory of robust estimation}.
\newblock University of California, Berkeley, 1968.

\bibitem[Hampel(1971)]{hampel1971general}
F.~R. Hampel.
\newblock A general qualitative definition of robustness.
\newblock \emph{The Annals of Mathematical Statistics}, 42\penalty0
  (6):\penalty0 1887--1896, 1971.

\bibitem[Howard and Ramdas(2022)]{howard2022sequential}
S.~R. Howard and A.~Ramdas.
\newblock Sequential estimation of quantiles with applications to {A/B} testing
  and best-arm identification.
\newblock \emph{Bernoulli}, 28\penalty0 (3):\penalty0 1704--1728, 2022.

\bibitem[Howard et~al.(2020)Howard, Ramdas, McAuliffe, and
  Sekhon]{howard2020time}
S.~R. Howard, A.~Ramdas, J.~McAuliffe, and J.~Sekhon.
\newblock Time-uniform {Chernoff} bounds via nonnegative supermartingales.
\newblock \emph{Probability Surveys}, 17:\penalty0 257--317, 2020.

\bibitem[Howard et~al.(2021)Howard, Ramdas, McAuliffe, and
  Sekhon]{howard2021time}
S.~R. Howard, A.~Ramdas, J.~McAuliffe, and J.~Sekhon.
\newblock Time-uniform, nonparametric, nonasymptotic confidence sequences.
\newblock \emph{The Annals of Statistics}, 49\penalty0 (2):\penalty0
  1055--1080, 2021.

\bibitem[Huber(1964)]{huber1964robust}
P.~J. Huber.
\newblock Robust estimation of a location parameter.
\newblock \emph{The Annals of Mathematical Statistics}, 35\penalty0
  (1):\penalty0 73--101, 1964.

\bibitem[Huber(1965)]{huber1965robust}
P.~J. Huber.
\newblock A robust version of the probability ratio test.
\newblock \emph{The Annals of Mathematical Statistics}, pages 1753--1758, 1965.

\bibitem[Huber(1968)]{huber1968robust}
P.~J. Huber.
\newblock Robust confidence limits.
\newblock \emph{Zeitschrift f{\"u}r Wahrscheinlichkeitstheorie und verwandte
  Gebiete}, 10\penalty0 (4):\penalty0 269--278, 1968.

\bibitem[Huber(1973)]{huber1973robust}
P.~J. Huber.
\newblock Robust regression: asymptotics, conjectures and {M}onte {C}arlo.
\newblock \emph{The Annals of Statistics}, pages 799--821, 1973.

\bibitem[Jamieson et~al.(2014)Jamieson, Malloy, Nowak, and
  Bubeck]{jamieson2014lil}
K.~Jamieson, M.~Malloy, R.~Nowak, and S.~Bubeck.
\newblock lil’ucb: An optimal exploration algorithm for multi-armed bandits.
\newblock In \emph{Conference on Learning Theory}, pages 423--439. PMLR, 2014.

\bibitem[Johari et~al.(2021)Johari, Pekelis, and Walsh]{johari2015always}
R.~Johari, L.~Pekelis, and D.~J. Walsh.
\newblock Always valid inference: Bringing sequential analysis to {A/B}
  testing.
\newblock \emph{Operations Research}, 70\penalty0 (3):\penalty0 1806--1821,
  2021.

\bibitem[Lai and Robbins(1985)]{lai1985asymptotically}
T.~L. Lai and H.~Robbins.
\newblock Asymptotically efficient adaptive allocation rules.
\newblock \emph{Advances in applied mathematics}, 6\penalty0 (1):\penalty0
  4--22, 1985.

\bibitem[Lecu{\'e} and Lerasle(2020)]{lecue2020robust}
G.~Lecu{\'e} and M.~Lerasle.
\newblock Robust machine learning by median-of-means: theory and practice.
\newblock \emph{The Annals of Statistics}, 48\penalty0 (2):\penalty0 906--931,
  2020.

\bibitem[Lerasle and Oliveira(2011)]{lerasle2011robust}
M.~Lerasle and R.~I. Oliveira.
\newblock Robust empirical mean estimators.
\newblock \emph{arXiv preprint arXiv:1112.3914}, 2011.

\bibitem[Li(2019)]{jlilecturenotes}
J.~Li.
\newblock Total variation, statistical models, and lower bounds, 2019.
\newblock URL \url{https://jerryzli.github.io/robust-ml-fall19/lec2.pdf}.
\newblock Lecture 2 in the Lecture notes of Robustness in Machine Learning (CSE
  599-M) taught at the University of Washington.

\bibitem[Lindon et~al.(2022)Lindon, Sanden, and Shirikian]{lindon2022rapid}
M.~Lindon, C.~Sanden, and V.~Shirikian.
\newblock Rapid regression detection in software deployments through sequential
  testing.
\newblock \emph{arXiv preprint arXiv:2205.14762}, 2022.

\bibitem[Lugosi and Mendelson(2021)]{lugosi2021robust}
G.~Lugosi and S.~Mendelson.
\newblock Robust multivariate mean estimation: the optimality of trimmed mean.
\newblock \emph{The Annals of Statistics}, 49\penalty0 (1):\penalty0 393--410,
  2021.

\bibitem[Maronna et~al.(2019)Maronna, Martin, Yohai, and
  Salibi{\'a}n-Barrera]{maronna2019robust}
R.~A. Maronna, R.~D. Martin, V.~J. Yohai, and M.~Salibi{\'a}n-Barrera.
\newblock \emph{Robust statistics: theory and methods (with {R})}.
\newblock John Wiley \& Sons, 2019.

\bibitem[Minsker and Ndaoud(2021)]{minsker2021robust}
S.~Minsker and M.~Ndaoud.
\newblock Robust and efficient mean estimation: an approach based on the
  properties of self-normalized sums.
\newblock \emph{Electronic Journal of Statistics}, 15\penalty0 (2):\penalty0
  6036--6070, 2021.

\bibitem[Prasad et~al.(2019)Prasad, Balakrishnan, and
  Ravikumar]{prasad2019unified}
A.~Prasad, S.~Balakrishnan, and P.~Ravikumar.
\newblock A unified approach to robust mean estimation.
\newblock \emph{arXiv preprint arXiv:1907.00927}, 2019.

\bibitem[Prasad et~al.(2020)Prasad, Balakrishnan, and
  Ravikumar]{prasad2020robust}
A.~Prasad, S.~Balakrishnan, and P.~Ravikumar.
\newblock A robust univariate mean estimator is all you need.
\newblock In \emph{International Conference on Artificial Intelligence and
  Statistics}, pages 4034--4044. PMLR, 2020.

\bibitem[Quang(1985)]{quang1985robust}
P.~X. Quang.
\newblock Robust sequential testing.
\newblock \emph{The Annals of Statistics}, pages 638--649, 1985.

\bibitem[Ramdas et~al.(2022{\natexlab{a}})Ramdas, Gr{\"u}nwald, Vovk, and
  Shafer]{ramdas2022game}
A.~Ramdas, P.~Gr{\"u}nwald, V.~Vovk, and G.~Shafer.
\newblock Game-theoretic statistics and safe anytime-valid inference.
\newblock \emph{arXiv preprint arXiv:2210.01948}, 2022{\natexlab{a}}.

\bibitem[Ramdas et~al.(2022{\natexlab{b}})Ramdas, Ruf, Larsson, and
  Koolen]{ramdas2022testing}
A.~Ramdas, J.~Ruf, M.~Larsson, and W.~M. Koolen.
\newblock Testing exchangeability: Fork-convexity, supermartingales and
  e-processes.
\newblock \emph{International Journal of Approximate Reasoning}, 141:\penalty0
  83--109, 2022{\natexlab{b}}.

\bibitem[Ronchetti and Huber(2009)]{ronchetti2009robust}
E.~M. Ronchetti and P.~J. Huber.
\newblock \emph{Robust statistics}.
\newblock John Wiley \& Sons, 2009.

\bibitem[Tukey(1960)]{tukey1960survey}
J.~W. Tukey.
\newblock A survey of sampling from contaminated distributions.
\newblock \emph{Contributions to probability and statistics}, pages 448--485,
  1960.

\bibitem[Ville(1939)]{ville1939etude}
J.~Ville.
\newblock Etude critique de la notion de collectif.
\newblock \emph{Bull. Amer. Math. Soc}, 45\penalty0 (11):\penalty0 824, 1939.

\bibitem[Wald(1945)]{wald1945sequential}
A.~Wald.
\newblock Sequential tests of statistical hypotheses.
\newblock \emph{The Annals of Mathematical Statistics}, 16\penalty0
  (2):\penalty0 117--186, 1945.

\bibitem[Wang and Ramdas(2022)]{wang2022catoni}
H.~Wang and A.~Ramdas.
\newblock {Catoni}-style confidence sequences for heavy-tailed mean estimation.
\newblock \emph{arXiv preprint arXiv:2202.01250}, 2022.

\bibitem[Wang et~al.(2021)Wang, Gurbuzbalaban, Zhu, Simsekli, and
  Erdogdu]{wang2021convergence}
H.~Wang, M.~Gurbuzbalaban, L.~Zhu, U.~Simsekli, and M.~A. Erdogdu.
\newblock Convergence rates of stochastic gradient descent under infinite noise
  variance.
\newblock \emph{Advances in Neural Information Processing Systems},
  34:\penalty0 18866--18877, 2021.

\bibitem[Waudby-Smith and Ramdas(2023)]{waudby2020estimating}
I.~Waudby-Smith and A.~Ramdas.
\newblock Estimating means of bounded random variables by betting.
\newblock \emph{Journal of the Royal Statistical Society: Series B
  (Methodological), to appear with discussion}, 2023.

\bibitem[Yang et~al.(2017)Yang, Ramdas, Jamieson, and
  Wainwright]{yang2017framework}
F.~Yang, A.~Ramdas, K.~G. Jamieson, and M.~J. Wainwright.
\newblock A framework for multi-{A}(rmed)/{B}(andit) testing with online {FDR}
  control.
\newblock \emph{Advances in Neural Information Processing Systems}, 30, 2017.

\end{thebibliography}

\newpage

%% appendix: can be appended after the main text

%%%%%%%%%%%%%%%

\appendix

\sectionallcap{Additional Theoretical Results and Proofs}\label{sec:pf}

\begin{lemma}\label{lem:dirac-mixture}
For every $\varepsilon \in (0, 1/2)$ and $\sigma > 0$, there exist $P_1, P_2 \in\mathcal{M}^2_{\sigma^2}$ such that
\begin{enumerate}
    \item $P_2 = (1-\varepsilon) P_1 + \varepsilon N$ for some $N \in \mathcal{M}$, and
    \item $|\mu(P_1) - \mu(P_2)| = \sigma \sqrt{\varepsilon}$.
\end{enumerate}
\end{lemma}
\begin{proof}[Proof of \cref{lem:dirac-mixture}]
We take $P_1 = \delta_{0}$ and $N = \delta_{\sigma \varepsilon^{-1/2}}$, the Dirac point mass at 0 and $\sigma \varepsilon^{-1/2} $ respectively. The fact that $P_1 \in \mathcal{M}^2_{\sigma^2}$ is obvious. Note that
\begin{equation}
    \mu(P_2) = (1-\varepsilon) \mu(P_1) + \varepsilon\mu(N) = \sigma \varepsilon^{1/2},
\end{equation}
and
\begin{equation}
    v_2(P_2) = (1-\varepsilon) \int (x - \sigma \varepsilon^{1/2})^2 \, \d P_1 + \varepsilon\int (x - \sigma \varepsilon^{1/2})^2 \, \d P_2 = (1-\varepsilon)\sigma^2.
\end{equation}
So $P_2 \in \mathcal{M}^2_{\sigma^2}$, concluding the proof.
\end{proof}

\begin{proof}[Proof of \cref{lem:lb}] Let $P_1$ and $P_2$ be the two distributions in \cref{lem:dirac-mixture}.
First applying the definition of RCS with $P_2 \in \tvball(P_1, \varepsilon)$, we have
\begin{equation}
    \Prw_{X_i \iid P_2} \left[ \forall t \in \mathbb N^+ , \  \mu(P_1) \in \CI_t (X_1, \dots, X_t) \right] \ge 1- \alpha.
\end{equation}
Then applying the definition of an RCS with $P_2 \in \tvball(P_2, \varepsilon)$, we get
\begin{equation}
    \Prw_{X_i \iid P_2} \left[ \forall t \in \mathbb N^+ , \  \mu(P_2) \in \CI_t (X_1, \dots, X_t) \right] \ge 1- \alpha.
\end{equation}
We see that, by a union bound,
\begin{equation}
    \Prw_{X_i \iid P_2} \left[ \forall t \in \mathbb N^+ , \  \{ \mu(P_1), \mu(P_2) \} \subseteq \CI_t (X_1, \dots, X_t) \right] \ge 1- 2 \alpha.
\end{equation}
This implies that
\begin{equation}
    \Prw_{X_i \iid P_2} \left[ \forall t \in \mathbb N^+ , \ \diam(\CI_t (X_1, \dots, X_t)) \ge \sigma \sqrt{\varepsilon} \right] \ge 1- 2 \alpha,
\end{equation}
as claimed. The case for RCIs is entirely analogous, by taking ``$\forall t \in \mathbb N^+$'' outside of ``$\Pr [ \, \cdot \, ]$''.
\end{proof}

\vspace{1em}

\begin{lemma}\label{lem:linear-bound-at-0} On $x \in [0, 1/7]$, we have
\begin{equation}
    0 \le \underbrace{\frac{x}{4} - 2\left( \frac{1}{1 + x/8 + 1.5x} - (1.5x + 1) \right)}_{A(x)} \le  \underbrace{\frac{x}{4} - 2\left( \frac{\exp(-x/4)}{1 +  x/8  + 1.5x} - (1.5x + 1) \right)}_{B(x)} \le 7x \le 1.
\end{equation}
\end{lemma}
\begin{proof}[Proof of \cref{lem:linear-bound-at-0}]
The first inequality is straightforward when writing $A(x)$ as $ \frac{3.25 x (1.625 x + 2)}{1.625 x + 1}$. The second inequality is trivial. To prove the third inequality, first note that $B(x) = 3.25 x - \frac{2 \exp(-x/4)}{1.625 x  + 1} + 2$; then note that
\begin{equation}
    2 \exp(-x/4) \ge 2 -0.5x \ge 2 - 0.5 x - 3.75\times 1.625 x^2 = (2-3.75 x)(1.625 x + 1),
\end{equation}
which clearly implies $B(x) \le 7x$. The fourth inequality is trivial. 
\end{proof}

\begin{proof}[Proof of \cref{thm:tight}]
Let $X_1, X_2, \dots \iid  Q \in \tvball(P, %2
\varepsilon)$.

Define
\begin{align}
    M_t(m) & = \prod_{i=1}^t \frac{ \exp \{ \phi(\lambda_i (X_i - m))   \}}{ 1+ \lambda_i(\mu(P) - m) + \frac{\lambda_i^2}{2} \left(\sigma^2 + (\mu(P) - m)^2 \right)  + 1.5 \varepsilon  } \label{eqn:M-of-m}
    \\ & =  \frac{\prod_{i=1}^t  \exp \{ \phi(\lambda (X_i - m))   \}}{\left( 1+ \lambda(\mu(P) - m) + \frac{\lambda^2}{2} \left(\sigma^2 + (\mu(P) - m)^2 \right)  + 1.5 \varepsilon  \right)^t}.
\end{align}
Note that
\begin{align}
    & \Exp \left[  \frac{ \exp \{ \phi(\lambda_t (X_t - m))  \}}{  1+ \lambda_t(\mu(P) - m) + \frac{\lambda_t^2}{2} \left(\sigma^2 + (\mu(P) - m)^2 \right)  + 1.5 \varepsilon  } \middle \vert \mathcal{F}_{t-1} \right]
    \\
    = & \frac{\Expw_{X_t \sim  Q}\left[ \exp \{ \phi(\lambda_t (X_t - m))  \} \right]  }{  1+ \lambda_t(\mu(P) - m) + \frac{\lambda_t^2}{2} \left(\sigma^2 + (\mu(P) - m)^2 \right)  + 1.5 \varepsilon  }
    \\
    \le & \frac{\Expw_{X_t \sim P }\left[ \exp \{ \phi(\lambda_t (X_t - m)) \}  \right] + 1.5 \varepsilon  }{ 1+ \lambda_t(\mu(P) - m) + \frac{\lambda_t^2}{2} \left(\sigma^2 + (\mu(P) - m)^2 \right)  + 1.5 \varepsilon   }
    \\
    \le & \frac{ \Expw_{X_t \sim P }\left[  1 + \lambda_t (X_t - m) + \lambda_t^2 (X_t - m)^2/2  \right] + 1.5 \varepsilon }{  1+ \lambda_t(\mu(P) - m) + \frac{\lambda_t^2}{2} \left(\sigma^2 + (\mu(P) - m)^2 \right)  + 1.5 \varepsilon  }
    \\
    = & \frac{ (1 + \lambda_t(\mu(P) - m) + \frac{\lambda_t^2}{2} \left(v_2(P) + (\mu(P) - m)^2)\right) + 1.5 \varepsilon  }{ 1+ \lambda_t(\mu(P) - m) + \frac{\lambda_t^2}{2} \left(\sigma^2 + (\mu(P) - m)^2 \right)  + 1.5 \varepsilon  }
    \\
    \le & \frac{(1 + \lambda_t(\mu(P) - m) + \frac{\lambda_t^2}{2} \left(\sigma^2 + (\mu(P) - m)^2)\right) + 1.5 \varepsilon  }{ 1+ \lambda_t(\mu(P) - m) + \frac{\lambda_t^2}{2} \left(\sigma^2 + (\mu(P) - m)^2 \right)  + 1.5 \varepsilon  }
   = 1.
\end{align}
Hence, $\{ M_t(m) \}$ is a supermartingale under $X_1, X_2, \dots \iid  Q$. We remark that $M_t(\mu(P))$ is just $\Mrc_t$ and when $\varepsilon = 0$ it is reduced to the non-robust case in our last paper.

Now that $\{ M_t(m) \}$ is a supermartingale with $M_0(m) = 1$, we see that $\Exp M_t(m) \le 1$; that is,
\begin{equation}
    \Exp \exp (f_t(m)) \le \left( 1+ \lambda(\mu(P) - m) + \frac{\lambda^2}{2} \left(\sigma^2 + (\mu(P) - m)^2 \right)  + 1.5 \varepsilon  \right)^t.
\end{equation}
Define the function
\begin{equation}
    B^+_t(m) = t \log \left(  1+ \lambda(\mu(P) - m) + \frac{\lambda^2}{2} \left(\sigma^2 + (\mu(P) - m)^2 \right)  + 1.5 \varepsilon  \right) + \log(2/\delta).
\end{equation}
By Markov's inequality,
\begin{equation}
    \forall m \in \mathbb R, \; \Pr[ f_t (m) \le B^+_t(m) ] \ge 1- \delta/2.
\end{equation}
Let $m = \pi_t$ be the smaller solution (whose existence is discussed soon) to the following \emph{quadratic} equation
\begin{equation}\label{eqn:upper-est}
    B_t^+(m) =  - \log(2/\alpha) - t \log (  1 + \lambda^2 \sigma^2 / 2 + 1.5 \varepsilon ).
\end{equation}
Then
\begin{equation}
     \Pr[f_t (\pi_t) \le  - \log(2/\alpha) - t \log (  1 + \lambda^2 \sigma^2 / 2 + 1.5 \varepsilon ) ] \ge 1 - \delta/2,
\end{equation}
\begin{equation}
    \Pr[\up(\CIRC_t) \le \pi_t] \ge 1 - \delta/2.
\end{equation}

Let us see how $\pi_t$ can exist. $\pi_t$ is a solution to a \emph{quadratic} equation. Hence it has closed-form expressions. Consider the equation~\eqref{eqn:upper-est},
\begin{equation}\label{eqn:upper-est-sorted}
   -\lambda(m- \mu(P)) + \frac{\lambda^2}{2} \left(\sigma^2 + (m- \mu(P))^2)\right)
   =  \left( \frac{4}{\alpha \delta} \right)^{-1/t} \frac{1}{1 + \lambda^2 \sigma^2 / 2 + 1.5 \varepsilon}   - (1+1.5 \varepsilon ).
\end{equation}
It has roots if and only if
\begin{equation}
    1 - \left( \lambda^2 \sigma^2 - 2 \left( \left( \frac{4}{\alpha \delta} \right)^{-1/t} \frac{1}{1 + \lambda^2 \sigma^2 / 2 + 1.5 \varepsilon}   - (1+1.5 \varepsilon ) \right) \right) \ge 0.
\end{equation}
Recall that we assume the following:
\begin{enumerate}
    \item $t > 4 \varepsilon^{-1} \log (4/\alpha \delta)  $, consequently $\e^{-\varepsilon/4} < \left( \frac{4}{\alpha \delta} \right)^{-1/t} < 1$,
    \item $\lambda^2 = \frac{\varepsilon}{4 \sigma^2}$,
    \item $0 < \varepsilon \le 1/7$.
\end{enumerate}
Hence we can apply \cref{lem:linear-bound-at-0} on $\varepsilon \in [0, 1/7]$ to get
\begin{equation}
    0 \le  \lambda^2 \sigma^2 - 2 \left( \left( \frac{4}{\alpha \delta} \right)^{-1/t} \frac{1}{1 + \lambda^2 \sigma^2 / 2 + 1.5 \varepsilon}   - (1+1.5 \varepsilon ) \right)  \le  7\varepsilon \le 1.
\end{equation}
So the smaller root exists and it satisfies
\begin{align}
    \pi_t &= \mu(P) + \frac{1 - \sqrt{ 1 - \left( \lambda^2 \sigma^2 - 2 \left( \left( \frac{4}{\alpha \delta} \right)^{-1/t} \frac{1}{1 + \lambda^2 \sigma^2 / 2 + 1.5 \varepsilon}   - (1+1.5 \varepsilon ) \right) \right) }}{\lambda}
    \\
    & \le \mu(P) + \frac{1 - \left[1 - \left(\lambda^2 \sigma^2 - 2 \left( \left( \frac{4}{\alpha \delta} \right)^{-1/t} \frac{1}{1 + \lambda^2 \sigma^2 / 2 + 1.5 \varepsilon}   - (1+1.5 \varepsilon ) \right)  \right) \right]}{\lambda}
    \\
    & = \mu(P) +  \frac{   \lambda^2 \sigma^2 - 2 \left( \left( \frac{4}{\alpha \delta} \right)^{-1/t} \frac{1}{1 + \lambda^2 \sigma^2 / 2 + 1.5 \varepsilon}   - (1+1.5 \varepsilon ) \right) }{\lambda}
    \\
    & \le \mu(P) + \frac{7\varepsilon}{\lambda}
    \\
    & = \mu(P) + 14 \sigma \sqrt{\varepsilon}.
\end{align}
We see that
\begin{equation}\label{eqn:first-boxed}
  \boxed{  \Pr[\up(\CIRC_t) \le  \mu(P) + 14 \sigma \sqrt{\varepsilon}] \ge 1 - \delta/2.}
\end{equation}

\vspace{0.25in}

Similarly, $\{ N_t(m) \}$ is a supermartingale under $X_1, X_2, \dots \iid  Q$, where
\begin{align}
     N_t(m) & = \prod_{i=1}^t \frac{ \exp \{ -\phi(\lambda_i (X_i - m))   \}}{ 1  - \lambda_i(\mu(P) - m) + \frac{\lambda_i^2}{2} \left(\sigma^2 + (\mu(P) - m)^2)\right)  + 1.5 \varepsilon  }
    \\ & =  \frac{\prod_{i=1}^t  \exp \{ -\phi(\lambda (X_i - m))   \}}{\left( 1  - \lambda(\mu(P) - m) + \frac{\lambda^2}{2} \left(\sigma^2 + (\mu(P) - m)^2)\right)  + 1.5 \varepsilon \right)^t}.
\end{align}
Define the function
\begin{equation}\label{eqn:lower-est}
    B^-_t(m) = -t \log \left(  1   -\lambda(\mu(P) - m) + \frac{\lambda^2}{2} \left(\sigma^2 + (\mu(P) - m)^2)\right)  + 1.5 \varepsilon \right) - \log(2/\delta).
\end{equation}
By Markov's inequality,
\begin{equation}
    \forall m \in \mathbb R, \; \Pr[ f_t (m) \ge B^-_t(m) ] \ge 1- \delta/2.
\end{equation}
Let $m = \rho_t$ be the larger solution to the following equation
\begin{equation}
    B_t^-(m) =   \log(2/\alpha) + t  \log \{  1 + \lambda^2 \sigma^2 / 2 + 1.5 \varepsilon \}.
\end{equation}
Then
\begin{equation}
    \Pr[\lw(\CIRC_t) \ge \rho_t] \ge 1 - \delta/2.
\end{equation}
Now, consider the equation~\eqref{eqn:lower-est},
\begin{equation}\label{eqn:lower-est-sorted}
    -\lambda(\mu(P)-m)+\frac{\lambda^2}{2}( \sigma^2 +  (\mu(P)-m)^2 ) =   \left( \frac{4}{\alpha \delta} \right)^{-1/t} \frac{1}{1 + \lambda^2 \sigma^2 / 2 + 1.5 \varepsilon}   - (1+1.5 \varepsilon) .
\end{equation}
The obvious isometry ($m-\mu(P) \leftrightarrow \mu(P) - m$) between~\eqref{eqn:upper-est-sorted} and~\eqref{eqn:lower-est-sorted} yields that $\rho_t \ge \mu(P) - 16\sigma \sqrt{\varepsilon}$. Therefore, we see that
\begin{equation}\label{eqn:second-boxed}
  \boxed{  \Pr[\lw(\CIRC_t) \ge  \mu(P) - 14 \sigma \sqrt{\varepsilon}] \ge 1 - \delta/2.}
\end{equation}
Combining two boxed inequalities via a union bound, we complete the proof.
\end{proof}

\begin{proof}[Proof of \cref{lem:repl-obliv}]
For any $P\in\mathcal{P}$ and any $Q \in \tvball(P, \varepsilon)$, consider $Z_1, \dots, Z_t \iid Q$. Recall that the TV metric between $P$ and $Q$ always equals the minimum of $\Prw_{(X,Y)\sim (P,Q) }[X \neq Y]$ over all coupling $(X,Y)$ with marginals $P$ and $Q$ respectively \citep{gibbs2002choosing}. Take the coupling $(P, Q)$ such that $\Prw_{(X,Y)\sim (P,Q) }[X \neq Y] = \tv(P, Q) \le \varepsilon$. Let $(X_1, Y_1), \dots, (X_t, Y_t) \iid (P,Q)$. Then,
\begin{align}
    & \Prw_{Z_i \iid Q} \left[ \chi(P) \in \CI_t(Z_1,\dots, Z_t) \right]
    \\
    = & \Prw_{(X_i, Y_i) \iid (P,Q)} \left[ \chi(P) \in \CI_t(Y_1,\dots, Y_t) \right]
    \\
    \ge & \Prw_{(X_i, Y_i) \iid (P,Q)} \left[ \exists c_0 \in \mathcal R_{2\varepsilon, t},  (Y_1,\dots,Y_t)=c_0(X_1,\dots,X_t) \text{ and }  \chi(P) \in \CI_t(c_0(X_1,\dots, X_t)) \right]
    \\
    \ge & \Prw_{(X_i, Y_i) \iid (P,Q)} \left[ \exists c_0 \in \mathcal R_{2\varepsilon, t},  (Y_1,\dots,Y_t)=c_0(X_1,\dots,X_t)\right] +  \Prw_{X_i \iid P} \left[ \forall c \in \mathcal R_{2\varepsilon, t}, \ \chi(P) \in \CI_t(c(X_1,\dots, X_t)) \right] - 1
    \\
    &\text{(bounding the first term by Hoeffding's inequality)}
    \\
    \ge & (1-\e^{-2t\varepsilon^2}) + (1-\alpha) - 1 = 1 - (\alpha + \e^{-2t\varepsilon^2}).
\end{align}
This completes the proof.
\end{proof}

\begin{proof}[Proof of \cref{cor:growth}] Recall the definition of $\{M_t(m)\}$ back in~\eqref{eqn:M-of-m}. Note that $\min(\CIRC_t)$ is the solution for $m$ of the equation $ M_t(m) = 1/\alpha $.

Under the assumptions of \cref{thm:tight},
inequality~\eqref{eqn:second-boxed} holds. Therefore, if $\mu_0 < \mu(P) - 14\sigma \sqrt{\varepsilon}$, $ \Pr[  \min(\CIRC_t) > \mu_0 ] \ge 1-  \delta/2$. Note that $M_t(m)$ is a decreasing function of $m$. This implies that $\Pr[ M_t(\mu_0) > 1/\alpha ] \ge 1 - \delta /2$.

Let us use the smallest $\alpha$ that satisfies the
assumption $t \ge 4 \varepsilon^{-1} \log(4/\alpha \delta)$ of \cref{thm:tight}. This would give
\begin{equation}
 \Prw_{X_i \sim Q} \left[  M_t(\mu_0) > \frac{\delta}{4} \exp\left(\frac{t \varepsilon}{4}\right) \right] \ge 1- \delta/2.
\end{equation}
The proof is complete as $M_t(\mu_0)$ is just the $\Mrc_t$ in \cref{cor:growth}.
\end{proof}

\begin{proof}[Proof of \cref{lem:rcatsm-infvar} and \cref{thm:rcs-p}]
Since $|\phi_p(x)| \le \log p$, we have $ 1/p \le \exp \{ \phi(\lambda_t (x - \mu(P))) \} \le p$. Note that
\begin{align}
    & \Exp \left[ \frac{ \exp \{ \phi_p(\lambda_t (X_t - \mu(P)))  \}}{ 1 + \lambda_t^p \kappa / p + (p-1/p) \varepsilon } \middle \vert \mathcal{F}_{t-1} \right]
    \\
     = & \frac{\Expw_{X_t \sim Q}\left[ \exp \{ \phi_p(\lambda_t (X_t - \mu(P))) \} \right]}{  1 +\lambda_t^p \kappa / p + (p-1/p) \varepsilon  }
    \\
      \le & \frac{ \Expw_{X_t \sim P }\left[ \exp \{ \phi_p(\lambda_t (X_t - \mu(P))) \}  \right] + (p-1/p) \varepsilon  }{  1 +\lambda_t^p \kappa / p + (p-1/p) \varepsilon  }
    \\
    \le & \frac{\Expw_{X_t \sim P }\left[  1 + \lambda_t (X_t - \mu(P)) + \lambda_t^p |X_t - \mu(P)|^p/p  \right] + (p-1/p)\varepsilon  }{  1 +\lambda_t^p \kappa / p + (p-1/p) \varepsilon  }
    \\
    = & \frac{ \Expw_{X_t \sim P }\left[  1 + \lambda_t^p |X_t - \mu(P)|^p/p  \right] + (p-1/p) \varepsilon  }{  1 +\lambda_t^p \kappa / 2 + (p-1/p) \varepsilon }
    \\
    = & \frac{ (1 + \lambda_t^p v_p(P)/p) + (p-1/p)  \varepsilon  }{ 1 +\lambda_t^p \kappa / p + (p-1/p)  \varepsilon  }  \le 1.
\end{align}
The first inequality above is due to~\eqref{eqn:tv-integral} and $Q \in \tvball(P, \varepsilon)$. Hence $\{ {\Mrc_t}^p \}$ is a supermartingale. The proof for $\{ {\Nrc_t}^p \}$ is analogous. Now apply Ville's inequality (\cref{lem:vil}) to $\{ {\Mrc_t}^p \}$. With probability at least $1-\alpha/2$, we have that $ \forall t \in \mathbb N^+$, 
\begin{equation}
    f_{pt}(\mu(P)) \le \sum_{i=1}^t \log\{ 1+ \lambda_i^p\kappa / p +(p-1/p)\varepsilon \} + \log(2/\alpha). 
\end{equation}
And to $\{ {\Nrc_t}^p \}$,
\begin{equation}
    -f_{pt}(\mu(P)) \le \sum_{i=1}^t \log\{ 1+ \lambda_i^p\kappa / p +(p-1/p)\varepsilon \} + \log(2/\alpha). 
\end{equation}
A union of the above two bounds concludes the proof.
\end{proof}

\vspace{1em}

\begin{lemma}\label{lem:p-poly-zero}
     Let $p \in (1,2]$, $C > 0$, and $
    g(y) = y^p - y + C$.  Suppose there is a $c > 0$ such that $C = \left(\frac{c}{(1+c)^p} \right)^{1/(p-1)}$. Then $g((1+c)C)=0$.
\end{lemma}
The above lemma is checked by direct substitution, so the proof is omitted.
% \begin{proof}[Proof of \cref{lem:p-poly-zero}]
%     $g((1+c)C) = (1+c)^p C^p - (1+c)C+C = C( (1+c)^p C^{p-1} - c ) = 0$.
% \end{proof}

\begin{lemma}\label{lem:p-poly-zero-2}
    Let $p \in (1,2]$, $A,B,C>0$, and $
    g(x) = Ax^p - Bx + C$. Suppose there is a $c>0$ such that $ A^{1/(p-1)} B^{-p/(p-1)} C =  \left(\frac{c}{(1+c)^p} \right)^{1/(p-1)}$.
    Then $g$ has a positive zero $(1+c) B^{-1} C $.
\end{lemma}

\begin{proof}[Proof of \cref{lem:p-poly-zero-2}]
    Substituting $x = (B/A)^{1/(p-1)}y$, the equation reads
    \begin{equation}
      y^p - y +  A^{1/(p-1)} B^{-p/(p-1)} C = 0.
    \end{equation}
    So there is a solution $ y_0 = (1+c)A^{1/(p-1)} B^{-p/(p-1)} C$ due to \cref{lem:p-poly-zero}. Correspondingly $ x_0=   (B/A)^{1/(p-1)} \cdot (1+c)A^{1/(p-1)} B^{-p/(p-1)} C = (1+c) B^{-1}C$.
\end{proof}

\begin{lemma}[One-dimensional special case of  Lemma 7 of \cite{wang2021convergence}, Appendix A]\label{lem:neurips-expand-lemma} Let $p \in (1,2]$. For any real $x$ and $y$, $|x+y|^p \le |x|^p + 4|y|^p + py|x|^{p-1} \sg(x)$.
\end{lemma}

\begin{lemma}\label{lem:linear-bound-at-0-infvar} For any $p \in (1,2]$ and $x > 0$, we have
\begin{equation}
   0 < \underbrace{ 1 + (p+3/p) x - \frac{1}{1 + p x}  }_{A(p,x)} <  1 + (p+3/p) x - \e^{-x} \frac{1}{1 + p x}  < 7x.
\end{equation}
\end{lemma}
\begin{proof}[Proof of \cref{lem:linear-bound-at-0-infvar}]
The first inequality is straightforward when writing $A(p,x)$ as $ \frac{ x (p^3 x + 2 p^2 +3px + 3)}{p(px+1)}$. The second inequality is trivial. The third inequality is equivalent to $\e^{-x} > (1+px)(1+(p+3/p-7)x) $. Since $p+3/p-7 < 0$ and $2p+3/p < 6$, we have $(1+px)(1+(p+3/p-7)x)  < 1 + (2p+3/p-7)x < 1 - x < \e^{-x}$.
\end{proof}

\vspace{2em}

\begin{proof}[Proof of \cref{thm:tight-infvar}]
Define
\begin{align}
    M_t^p(m) = \frac{\prod_{i=1}^t  \exp \{ \phi_p(\lambda_i (X_i - m))   \}}{ \prod_{i=1}^t \left( 1+ \lambda_i(\mu(P) - m) + \frac{\lambda_i^p}{p} \left(4\kappa + |\mu(P) - m|^p \right)  + (p-1/p) \varepsilon  \right)}.
\end{align}
Then
\begin{align}
    &\Exp\left[ \frac{  \exp \{ \phi_p(\lambda_t (X_t - m))   \}}{ 1+ \lambda_t(\mu(P) - m) + \frac{\lambda_t^p}{p} \left(4\kappa + |\mu(P) - m|^p \right)  + (p-1/p) \varepsilon } \middle \vert \mathcal{F}_{t-1}  \right]
    \\
    = & \frac{  \Expw_{X_t \sim Q} [\exp \{ \phi_p(\lambda_t (X_t - m))   \}]}{ 1+ \lambda_t(\mu(P) - m) + \frac{\lambda_t^p}{p} \left(4\kappa + |\mu(P) - m|^p \right)  + (p-1/p) \varepsilon }
    \\
    \le & \frac{  \Expw_{X_t \sim P} [\exp \{ \phi_p(\lambda_t (X_t - m))   \}] + (p-1/p)\varepsilon}{ 1+ \lambda_t(\mu(P) - m) + \frac{\lambda_t^p}{p} \left(4\kappa + |\mu(P) - m|^p \right)  + (p-1/p) \varepsilon }
    \\
    \le &  \frac{  \Expw_{X_t \sim P} [1 + \lambda_t (X_t - m) + \lambda_t^p |X_t - m|^p/p] + (p-1/p)\varepsilon}{ 1+ \lambda_t(\mu(P) - m) + \frac{\lambda_t^p}{p} \left(4\kappa + |\mu(P) - m|^p \right)  + (p-1/p) \varepsilon }
    \\
    & \text{(by \cref{lem:neurips-expand-lemma})}
    \\
    \le &  \tfrac{  \Expw_{X_t \sim P} [1 + \lambda_t (X_t - m) + (\lambda_t^p/p)( |\mu(P) - m|^p + 4|X_t - \mu(P)|^p + p(X_t - \mu(P))|\mu(P) - m|^{p-1} \sg(\mu(P) - m) ) ] + (p-1/p)\varepsilon}{ 1+ \lambda_t(\mu(P) - m) + \frac{\lambda_t^p}{p} \left(4\kappa + |\mu(P) - m|^p \right)  + (p-1/p) \varepsilon } 
    \\
    \le & \frac{1 + \lambda_t(\mu(P) - m) + (\lambda_t^p/p) ( |\mu(P) - m|^p +  4\kappa  ) + (p-1/p)\varepsilon  }{ 1+ \lambda_t(\mu(P) - m) + \frac{\lambda_t^p}{p} \left(4\kappa + |\mu(P) - m|^p \right)  + (p-1/p) \varepsilon} = 1.
\end{align}

Thus $\{ M_t^p(m) \}$ is a nonnegative supermartingale issued at 1.

When $\lambda_1 = \dots = \lambda$, we see that
\begin{equation}
    \Exp \exp(f_{pt}(m)) \le \left( 1+ \lambda(\mu(P) - m) + \frac{\lambda^p}{p} \left(4\kappa + |\mu(P) - m|^p \right)  + (p-1/p) \varepsilon  \right)^t.
\end{equation}
Define the function
\begin{equation}
    B^+_{pt}(m) = t \log \left( 1+ \lambda(\mu(P) - m) + \frac{\lambda^p}{p} \left(4\kappa + |\mu(P) - m|^p \right)  + (p-1/p) \varepsilon  \right) + \log(2/\delta).
\end{equation}
Markov's inequality yields
\begin{equation}
   \forall m \in \mathbb R,  \Pr [ f_{pt}(m) \le  B^+_{pt}(m)  ] \ge 1- \delta/2.
\end{equation}
Suppose $m = \pi_t$ is a solution in $(\mu, \infty)$ (existence to be discussed soon) to the equation
\begin{equation}\label{eqn:inf-var-upper-eqn}
     B^+_{pt}(m) = -\log(2/\alpha) - t\log(1 + \lambda^p \kappa / p + (p-1/p)\varepsilon).
\end{equation}
Then
\begin{gather}
    \Pr [ f_{pt}(\pi_t) \le -\log(2/\alpha) - t\log(1 + \lambda^p \kappa / p + (p-1/p)\varepsilon)  ] \ge 1- \delta/2
    \\
    \Pr[ \max({\CIRC_t}^p) \le \pi_t ] \ge 1- \delta/2.
\end{gather}

Let us see how $\pi_t$ can exist. The equation that $\pi_t$ would satisfy,~\eqref{eqn:inf-var-upper-eqn}, expands into
\begin{gather}
    1 - \lambda(m -\mu(P)) + \frac{\lambda^p}{p} \left(4\kappa + | m - \mu(P) |^p \right)  + (p-1/p) \varepsilon   = \left( \frac{4}{\alpha \delta} \right)^{-1/t} \frac{1}{1 + \lambda^p \kappa / p + (p-1/p)\varepsilon};
    \\
    \frac{\lambda^p}{p} |m - \mu(P)|^p - \lambda(m - \mu(P)) + \underbrace{ 1 + 4\lambda^p \kappa/p + (p-1/p)\varepsilon - \left( \frac{4}{\alpha \delta} \right)^{-1/t} \frac{1}{1 + \lambda^p \kappa / p + (p-1/p)\varepsilon} }_{\mathfrak{C}} = 0
\end{gather}
So, if we can make sure that there is a $c > 0$ such that
\begin{equation}\label{eqn:cC-condition}
  \underbrace{ (\lambda^p/p)^{1/(p-1)} \lambda^{-p/(p-1)} }_{p^{-1/(p-1)}} \mathfrak C =   \left(\frac{c}{(1+c)^p} \right)^{1/(p-1)}
\end{equation}
then according to \cref{lem:p-poly-zero-2} such $\pi_t$ exists and
\begin{equation}
   \pi_t =  \mu(P)  + (1+c) \lambda^{-1} \mathfrak C.
\end{equation}

To achieve this, recall that we have the following assumptions:
\begin{enumerate}
    \item $t \ge \varepsilon^{-1} \log(4/\alpha \delta)$. So $\e^{-\varepsilon} \le \left( \frac{4}{\alpha\delta} \right)^{-1/t} < 1$.
    \item $\lambda^p \kappa = \varepsilon$.
    \item $\varepsilon \in (0, \frac{p-1}{7p}]$.
\end{enumerate}

Then, we can bound $\mathfrak C =  1 + 4\lambda^p \kappa/p + (p-1/p)\varepsilon - \left( \frac{4}{\alpha \delta} \right)^{-1/t} \frac{1}{1 + \lambda^p \kappa / p + (p-1/p)\varepsilon}$ by functions of $\varepsilon$:
\begin{equation}
 1 + (p+3/p) \varepsilon - \frac{1}{1 + p \varepsilon}  \le \mathfrak C \le 1 + (p+3/p) \varepsilon  - \e^{-\varepsilon} \frac{1}{1 + p \varepsilon}.
\end{equation}
Applying \cref{lem:linear-bound-at-0-infvar}, we have $0 < \mathfrak C < 7 \varepsilon \le \frac{p-1}{p} $. 

 By elementary calculus, the function $ \frac{x}{(1+x)^p} $ ($x>0$) takes its maximum $J = \frac{(p-1)^{p-1}}{p^p}$ at $x_p = \frac{1}{p-1}$. Since $C \le \frac{p-1}{p} = (pJ)^{1/(p-1)}$, there is a $c \in (0, x_p]$ that satisfies $\mathfrak C = \left( \frac{pc}{(1+c)^p} \right)^{1/(p-1)}$, which is exactly the condition~\eqref{eqn:cC-condition}.

Therefore, $\pi_t$ exists and $\pi_t = \mu(p) + (1+c)\lambda^{-1} \mathfrak C \le \mu(P) + 7(1+x_p) \kappa^{1/p} \varepsilon^{(p-1)/p} $. It then follows that
\begin{equation}
    \Pr[ \max({\CIRC_t}^p) \le \mu(P) +  7 (1+x_p) \kappa^{1/p} \varepsilon^{(p-1)/p} ] \ge 1- \delta/2.
\end{equation}

The other direction, $\Pr[ \min({\CIRC_t}^p) \ge \mu(P) - 7 (1+x_p) \kappa^{1/p} \varepsilon^{(p-1)/p} ] \ge 1- \delta/2$, is entirely analogous. So we have
\begin{equation}\label{eqn:p-var-con}
    \Pr[ \diam({\CIRC_t}^p) \le  14 (1+x_p) \kappa^{1/p} \varepsilon^{(p-1)/p} ] \ge 1- \delta.
\end{equation}
Finally, recall that $x_p = \frac{1}{p-1}$. So the bound~\eqref{eqn:p-var-con} is exactly the one stated in the \cref{thm:tight-infvar}.
\end{proof}

\begin{proof}[Proof of \cref{thm:bet}]
%    For $\{K_t\}$, since $|X_t - \mu(P)| \le 1$ and $|\lambda_t| \le \frac{1}{2\varepsilon + 1}$, we have
 %   \begin{equation}
  %      0 \le  \frac{1}{2\varepsilon + 1} + \lambda_t(X_t - \mu(P)) \le \frac{2}{2\varepsilon + 1},
  %  \end{equation}
   % from which nonnegativity follows.
   % To see that it is a supermartingale, note that
    %\begin{align}
     %     \Expw_{X_t \sim Q} \left[\frac{1}{2\varepsilon + 1} + \lambda_t(X_t - \mu(P))\right]
     %   \le &  \Expw_{X_t \sim P} \left[\frac{1}{2\varepsilon + 1} + \lambda_t(X_t - \mu(P))\right] + \frac{2}{2\varepsilon + 1} \cdot \varepsilon
      %  \\
     %   = & \frac{1}{2\varepsilon + 1} + \frac{2\varepsilon}{2\varepsilon + 1} = 1.
    % \end{align}
    % Therefore $\{K_t\}$ is a nonnegative supermartingale.
    
    %For $\{L_t\}$, 
    First, each $1 + \lambda_i(X_t  - \mu(P)) -\varepsilon|\lambda_t|$ is nonnegative since
    \begin{equation}
        \varepsilon \le \frac{1}{|\lambda_t|} - 1  \le \frac{1}{|\lambda_t|} + X_t - \mu(P) =  \frac{1 + \lambda_t(X_t  - \mu(P))}{|\lambda_t|}.
    \end{equation}
   Further,
    \begin{align}
        &\Expw_{X_t \sim Q}[  1 + \lambda_t(X_t  - \mu(P)) -\varepsilon|\lambda_t| ] =1 + \lambda_t\left(\Expw_{X_t\sim Q}[X_t] - \mu(P)\right) -\varepsilon|\lambda_t| 
        \\
        \le & 1 + \lambda_t\left(\Expw_{X_t\sim P}[X_t] + \varepsilon - \mu(P)\right) -\varepsilon|\lambda_t| 
        \\
        = & 1 + \lambda_t\varepsilon  - \varepsilon|\lambda_t|  \le 1.
    \end{align}
    Therefore $\{ L_t \}$ is a nonnegative supermartingale.
\end{proof}

%%%%% addition during/after rebuttal

%\section{Extras}
%% todo: find dist that makes E exactly 1. Prop 4 of the annals paper
%% todo: add changing Qt and eps
%% bounded mult increment -> our approach e.g. betting smg, symmetric 
% \prod_i (1+ tanh (X_i)) / (1+2e) 

%\subsection{Tightness of \cref{lem:rcatsm}}

%The supermartingales in \cref{lem:rcatsm} is tight in the following (weak) sense. When $\varepsilon = 0$, for any $\sigma > 0$, $c\in (0, 1)$, $P \in \mathcal{M}_{\sigma^2}^2$, there exists a $\lambda > 0$, such that
%\begin{align}
 %   \Expw_{X \sim P} \left[ \frac{ \exp \{ \phi(\lambda (X - \mu(P)))  \}}{ 1 + \lambda^2 \sigma^2 / 2  } \right] > 1 - c.
%\end{align}

%(Proof: It might just be a continuity argument as $\lambda \to 0$ the quantity is 1)

\end{document}